\documentclass[12pt,reqno]{amsart}
\usepackage{amssymb}
\usepackage{amsfonts}
\usepackage{amsmath,latexsym,amssymb,amsfonts,amsbsy, amsthm}
\usepackage[usenames]{color}
\usepackage{xspace,colortbl}
\usepackage{mathrsfs}
\usepackage[colorlinks,linkcolor=blue,citecolor=blue]{hyperref}
\usepackage[titletoc]{appendix}
\allowdisplaybreaks

\newcommand{\beq}{\begin{equation}}
\newcommand{\eeq}{\end{equation}}
\newcommand{\ben}{\begin{eqnarray}}
\newcommand{\een}{\end{eqnarray}}
\newcommand{\beno}{\begin{eqnarray*}}
\newcommand{\eeno}{\end{eqnarray*}}
\newcommand{\R}{\mathbb{R}}
\newtheorem{thm}{Theorem}[section]

\newtheorem{lem}[thm]{Lemma}
\newtheorem{prop}[thm]{Proposition}
\newtheorem{coro}[thm]{Corollary}
\newtheorem{rmk}[thm]{Remark}

\title[Bistable or combustion   fronts]{Lipschitz property of bistable or combustion   fronts and its applications}

\author[K. Wang]{Kelei Wang$^\dag$}
\address{School of Mathematics and Statistics\\ Wuhan University\\
Wuhan 430072, China}
\email{wangkelei@whu.edu.cn}



\begin{document}

\begin{abstract}
For a class of reaction-diffusion equations  describing propagation phenomena, we prove that for any   entire solution $u$,  the level set $\{u=\lambda\}$ is a Lipschitz graph in the time direction  if $\lambda$ is close to $1$. Under a further assumption that $u$ connects $0$ and $1$, it is shown that all level sets are Lipschitz graphs. By a blowing down analysis,   the large scale motion law for these  level sets and a characterization of the minimal speed for travelling waves are also given.
\end{abstract}

\maketitle

%\tableofcontents

%%%%%%%%%%%%%%%%%%%%%%%%%%%%%%%%%%%%%%%%%%%%%%%%%%%%%%%%%%%%%%%%%%%%%%%%
\section{Introduction}\label{sec introduction}
\setcounter{equation}{0}
%%%%%%%%%%%%%%%%%%%%%%%%%%%%%%%%%%%%%%%%%%%%%%%%%%%%%%%%%%%%%%%%%%%%%%%%

\subsection{Lipschitz property for level sets}
Consider a smooth, entire solution $u$ to the reaction-diffusion equation
 \begin{equation}\label{eqn}
 \partial_tu-\Delta u=f(u), \quad 0<u<1 \quad \mbox{in } ~~\R^n\times \R.
 \end{equation}
In this paper we are mainly interested in the Lipschitz property of the level sets $\{u=\lambda\}$ and their geometric motion law at large scales.

Our  main hypothesis on $f$ are
\begin{description}
  \item [{\bf (F1)}] $f\in \mbox{Lip}([0,1])$,  $f(0)=f(1)=0$ and $f\in C^{1,\alpha}([0,\gamma)\cup(1-\gamma,1])$ for some $\alpha,\gamma \in(0,1)$;
  \item [{\bf (F2)}] $f^\prime(1)<0$;
  \item  [{\bf (F3)}] $\int_{0}^{1}f(u)du>0$.
\end{description}
Sometimes we  also need
\begin{description}
  \item[{\bf (F4)}]   there exists a $\theta\in(0,1)$ such that $f>0$ in $(\theta,1)$, and either $f<0$ in $(0,\theta)$ with $f^\prime(0)<0$, or $f\equiv 0$ in $(0,\theta)$.
\end{description}
Typical examples are the bistable  nonlinearity  $f(u)=u(u-\theta)(1-u)$ with $\theta<1/2$ and the combustion nonlinearity.

The reaction-diffusion equation \eqref{eqn} is used in the modelling of biological propagation phenomena, see Aronson and Weinberger \cite{Aronson1978multidimensional}. Entire solutions have  been studied by many people since the work of Hamel and Nadirashvili \cite{Hamel1999entire,Hamel2001entire}, see \cite{Hamel2016monotonicity, Hamel2016bistable, Hamel2016Fisher-KPP,Bu2019transition} for the homogeneous case. There is also a large literature devoted to the study of heterogeneous cases. In particular,  a very general notion of travelling fronts, \emph{transition fronts},  was introduced by Berestycki and Hamel in \cite{Hamel2007generalized, Hamel2012transition}. %Roquejoffre2009stability,,Wang2020stability

The geometry of an entire solution is complicated in general.
To study the Lipschitz property of $\{u=\lambda\}$, we introduce   some assumptions on the entire solution $u$. The first one is
\begin{description}
   \item  [{\bf(H1)}] For any $t\in\R$, $\sup_{x\in\R^n}u(x,t)=1$.
 \end{description}
Under this assumption we prove
\begin{thm}[Half Lipschitz property for entire solutions]\label{main result 3}
Suppose $f$ satisfies {\bf (F1-F3)}. There exists a $b_0\in(0,1)$ such that, if an entire solution  $u$ satisfies {\bf(H1)}, then for any $\lambda\in[1-b_0,1)$, $\{u=\lambda\}=\{t=h_\lambda(x)\}$  is a globally Lipschitz graph on $\R^n$.%, and $h_\lambda$ is continuous, increasing in $\lambda$.
\end{thm}
In general, if $\lambda$ is close to $0$, $\{u=\lambda\}$ does not satisfy this Lipschitz property, see the example given after Theorem \ref{main result 1}.   In order to establish the full Lipschitz property, we need more assumptions. A natural one is
\begin{description}
   \item  [{\bf(H2)}] $u\to 0$ uniformly as  $\mbox{dist}((x,t),\{u\geq 1-b_0\})\to+\infty$.
\end{description}
Here $\mbox{dist}$ denotes the standard Euclidean distance on $\R^n\times\R$.

\begin{thm}[Full Lipschitz property for entire solutions]\label{main result 4}
Suppose $f$ satisfies {\bf (F1-F4)}, and $u$ is an entire solution satisfying {\bf (H1-H2)}.  Then for any $\lambda\in(0,1)$, $\{u=\lambda\}=\{t=h_\lambda(x)\}$  is a globally Lipschitz graph on $\R^n$.
\end{thm}

The proof of Theorem \ref{main result 3} relies on the propagation phenomena (see  Aronson and Weinberger \cite{Aronson1978multidimensional}) in \eqref{eqn}. Roughly speaking, by {\bf(F3)}, $1$ represents a more stable state than $0$, so $\{u\approx 1\}$ will invade $\{u\approx0\}$. This gives us a  cone of monotonicity at large scales,  see Lemma \ref{lem forward Lip} for the precise statement. Although this only implies a Lipschitz property for $\{u=\lambda\}$ at large scales, it can be propagated to a real Lipschitz property by utilizing some estimates on positive solutions to the linear parabolic equation
\[\partial_tw-\Delta w=f^\prime(1)w.\]
Here the nondegeneracy condition $f^\prime(1)<0$ will be crucial for this argument.

 After establishing the Lipschitz propriety of $\{u=\lambda\}$ for $\lambda$ close to $1$, under the hypothesis {\bf(H2)}, we can apply the maximum principle and sliding method (in the time direction, cf. Guo and Hamel \cite{Hamel2016monotonicity}) to extend this Lipschitz property backwardly in time, which is Theorem \ref{main result 4}.

\subsection{Travelling wave solutions}
The solution $u$ is a travelling wave in the direction $-e_n$ and with speed $\kappa>0$, if there exists a function $v\in C^2(\R^n)$ such that
\[u(x,t)=v(x+\kappa te_n).\]
Here $v$ satisfies the elliptic equation
\begin{equation}\label{travelling wave eqn}
  -\Delta v+\kappa \partial_nv=f(v) \quad \mbox{in }~~ \R^n.
\end{equation}
 Among the class of travelling wave solutions,  the one dimensional travelling wave is of particular importance.
By \cite[Theorem 4.1]{Aronson1978multidimensional}, under the hypothesis {\bf (F1-F4)}, there exists a  unique constant $\kappa_\ast>0$ and a unique (up to a translation)
 solution to the one dimensional problem
\begin{equation}\label{1D wave}
  -g^{\prime\prime}(t)+\kappa_\ast g^\prime(t)=f(g(t)), \quad  g(-\infty)=0, \quad g(+\infty)=1.
\end{equation}

Theorem \ref{main result 3} applied to $v$ gives
\begin{thm}[Half Lipschitz property for travelling waves]\label{main result 1}
Suppose $f$ satisfies {\bf (F1-F3)} and $v$ is an entire solution of \eqref{travelling wave eqn}, satisfying $\sup_{\R^n}v=1$.
For any $\lambda\in[1-b_0,1)$, $\{v=\lambda\}=\{x_n=h_\lambda(x^\prime), x^\prime\in\R^{n-1}\}$  is a globally Lipschitz graph on $\R^{n-1}$.
\end{thm}
As in the entire solution case, in general, this property does not hold for   level sets $\{v=\lambda\}$ with $\lambda$  close to $0$. For example, in  Hamel and Roquejoffre \cite{Hamel2011heteroclinic}, it is shown that when $n=2$, there exist solutions $v$ of \eqref{eqn}, which is monotone in $x_1$ and satisfies
\[
\left\{\begin{aligned}
&v(x_1,x_2)\to 1 \quad \mbox{uniformly as } x_1\to+\infty.\\
&v(x_1,x_2)\to \varphi(x_2) \quad \mbox{locally uniformly as}~ x_1\to-\infty,
\end{aligned}\right.
\]
where $\varphi$ is an $L$-periodic solution (for some $L>0$) of
\[-\varphi^{\prime\prime}=f(\varphi)  \quad \mbox{in } \R.\]
Hence when $\lambda$ is close to $0$, $\{v=\lambda\}$ is the graph of an $L$-periodic function $h_\lambda$, satisfying $h_\lambda(kL)=-\infty$ for any $k\in\mathbb{Z}$. Clearly it cannot be a globally Lipschitz graph.

As in the entire solution case, in order to get the  Lipschitz property for all level sets, we need more assumptions. A natural one is
\begin{thm}[Full Lipschitz property for travelling waves]\label{main result 2}
Suppose $f$ satisfies {\bf (F1-F4)}, $v$ is an entire solution of \eqref{travelling wave eqn}, satisfying $\sup_{\R^n}v=1$ and
\begin{equation}\label{assumption}
  v(x)\to 0 \quad \mbox{uniformly as } \mbox{dist}(x,\{v\geq 1-b_0\})\to+\infty.
\end{equation}
 Then for any $\lambda\in(0,1)$, $\{v=\lambda\}=\{x_n=h_\lambda(x^\prime)\}$  is a globally Lipschitz graph on $\R^{n-1}$.
\end{thm}
This theorem also holds for the monostable case,  that is, instead of {\bf(F4)}, we assume
\begin{description}
  \item[{\bf (F4$^\prime$)}]    $f>0$ in $(0,1)$, and  $f^\prime(0)>0$.
\end{description}
The assumption \eqref{assumption} holds automatically in the monostable case. Hence we get a small improvement on  the same Lipschitz property for all level sets  proved in \cite{Hamel2001entire}, where they require the nonlinearity $f$ to be concave. However, we do not know how to prove the parabolic case, see discussions in Subsection \ref{subsec monostable case}.

Existence, qualitative properties and classification of solutions to \eqref{travelling wave eqn} with Lipschitz level sets have been studied by many people, see \cite{Hamel2000conical,Hamel2005conical, Hamel2006classification, Taniguchi2005stability, Taniguchi2006stability, Taniguchi2007pyramidal,Taniguchi2009,Taniguchi2011pyramidal,Taniguchi2012fronts,Taniguchi2015convex}.

\subsection{Blowing down limits}
Once we know level sets of $u$ are Lipschitz graphs, we would like to study their large scale structures. Take a $b\in(0,1)$ such that
$\{u=b\}=\{t=h(x)\}$ is a globally Lipschitz graph on $\R^n$. For any $\lambda>0$, let
\[h_\lambda(x):=\frac{1}{\lambda}h(\lambda x).\]
They are uniformly Lipschitz. Therefore for any $\lambda_i\to\infty$, there exist a subsequence (not relabelling) such  that $h_{\lambda_i}$ converges to $h_\infty$ in $C_{loc}(\R^n)$. (This limit may depend on the choice of subsequences.)

We have the following characterization of $h_\infty$.
\begin{thm}\label{thm blowing down limit}
  Under the assumptions of Theorem \ref{main result 4}, the blowing down limit $h_\infty$ is a viscosity solution of
  \begin{equation}\label{limit eqn 1}
|\nabla h_\infty|^2-\kappa_\ast^{-2}=0 \quad \mbox{in } \quad \R^{n}.
  \end{equation}
\end{thm}

\begin{rmk}[Level set formulation]
Equation \eqref{limit eqn 1} is the level set formulation of the geometric motion equation for the family of hypersurfaces $\Sigma(t):=\{x:h_\infty(x)=t\}$,
\begin{equation}\label{geometric motion}
  V_{\Sigma(t)}=\kappa_\ast\nu_{\Sigma(t)}.
\end{equation}
Here $\nu_{\Sigma(t)}=-\nabla h_\infty/|\nabla h_\infty|$ is the unit normal vector of $\Sigma(t)$. See Fife \cite[Chapter 1]{Fife1988CBMS} for a formal derivation of this equation.

The equation \eqref{limit eqn 1} also corresponds to the fact that the global mean speed of transition fronts equals $\kappa_\ast$, see Hamel \cite{Hamel2016bistable}.
\end{rmk}
\begin{rmk}\label{rmk representation for level set limit}
Because $h_\infty(0)=0$, the following representation formula holds for $h_\infty$ (see for example Monneau,  Roquejoffre and   Roussier-Michon \cite[Section 2]{Monneau2013graph}): there exists a closed set $\Xi\subset\mathbb{S}^{n-1}$ such that
\[ h_\infty(x)=\inf_{\xi\in\Xi}\xi\cdot x.\]
 As a consequence, $h_\infty$ is concave and $1$-homogeneous.
\end{rmk}

 The connection between reaction-diffusion equations and motion by mean curvatures in the framework of viscosity solutions has been explored by many people in 1980s and 1990s. In particular, the asymptotic behavior of solutions to the Cauchy problem for \eqref{eqn} has been studied by Barles, Bronsard, Evans, Soner and Souganidis in \cite{Evans1989optics,Barles1990wavefront,Barles1992front,Barles1994asymptotic}, in the framework of Hamilton-Jacobi equation and level set motions. We use the same idea, but now for the study of entire solutions of \eqref{eqn} (in the spirit of \cite{Hamel1999entire, Hamel2001entire}), where we are free to perform scalings  to study the large scale structure of  entire solutions.

From this blowing down analysis we also get a characterization of the minimal speed.
\begin{thm}\label{thm minimal speed}
Suppose $f$ satisfies {\bf (F1-F4)}, $v$ is an entire solution of \eqref{travelling wave eqn}, satisfying $\sup_{\R^n}v=1$ and
\eqref{assumption}. Then  $\kappa\geq\kappa_\ast$.

Furthermore, if $\kappa=\kappa_\ast$,  there exists a constant $t\in\R$ such that
\[v(x)\equiv g(x_n+t)  \quad \mbox{in }~~ \R^n.\]
\end{thm}

\subsection{Further problems}\label{sec discussion}

To put our results in a wide perspective, here we  mention some further problems about  \eqref{eqn} and \eqref{travelling wave eqn}. Some of these problems are well known to experts in this field.

{\bf Problem 1.} Extend results in this paper to the monostable case.

{\bf Problem 2.} Theorem \ref{thm blowing down limit} gives only the main order term of the front motion law. The next order term has been formally derived in Fife \cite{Fife1988CBMS}. Using the language of viscosity solutions,  the family of  hypersurfaces $\Sigma(t):=\{u(t)=1/2\}$ should be an approximate viscosity solution at large scales (in the sense of Savin \cite{Savin2009DeGiorgi,Savin2017DeGiorgi}) of  the forced mean curvature flow
\begin{equation}\label{geometric motion refined}
 V_{\Sigma(t)}=\left[\kappa_\ast -H_{\Sigma(t)}\right]\nu_{\Sigma(t)}.
\end{equation}
Here $H_{\Sigma(t)}$ denotes the mean curvature of $\Sigma(t)$.

{\bf Problem 3.}  In \cite{Hamel2001entire}, Hamel and  Nadirashvili proposed a conjecture about the classification of entire solutions. For travelling wave solutions in the bistable and combustion case, this conjecture may be broken into two steps:
\begin{enumerate}
  \item There exists a one to one correspondence between solutions of \eqref{travelling wave eqn} and solutions of
  \begin{equation}\label{travelling wave for forced MCF}
    \mbox{div}\left(\frac{\nabla h}{\sqrt{1+|\nabla h|^2}}\right)=\kappa_\ast-\frac{\kappa}{\sqrt{1+|\nabla h|^2}}.
  \end{equation}
This is the travelling wave equation   of \eqref{geometric motion refined}, see \cite{Monneau2013graph} for a discussion on this equation.
  \item There exists a one to one correspondence between solutions of \eqref{travelling wave for forced MCF} and nonnegative Borel measures on $\mathbb{S}^{n-1}$.
\end{enumerate}

{\bf Problem 4.} In view of the above discussion and Taniguchi's theorem in \cite{Taniguchi2015convex}, a not so ambitious question is if the reverse of Theorem \ref{thm blowing down limit} is true, that is, given a homogeneous viscosity solution $h_\infty$ of \eqref{limit eqn 1}, does there exist an entire solution of \eqref{eqn} so that its level set $\{u=1/2\}$ is asymptotic to $\{t=h_\infty(x)\}$?

\subsection{Notations and organization of the paper}
Throughout the paper we keep the following conventions.
\begin{itemize}
  \item  We use $C$ (large) and $c$ (small) to denote various universal constants, which could be different from line to line.
  \item  The parabolic boundary of a domain $\Omega\subset\R^n\times\R$  is denoted by $\partial^p\Omega$.
  \item A function $u\in C^{2,1}(\R^n\times\R)$ if it is $C^2$ in $x$-variables and $C^1$ in $t$-variable.
\end{itemize}

%The organization of this paper can be seen from the table of contents.
The remaining part of this paper is organized as follows. In Section \ref{sec propagation} we study the propagation phenomena in \eqref{eqn} and use this to prove Theorem \ref{main result 3}. In Section \ref{sec case near 0} we prove Theorem \ref{main result 4} by the sliding method. An elliptic Harnack inequality is established in Section \ref{sec elliptic Harnack}. In Section \ref{sec blowing down} we perform the blowing down analysis. In Section \ref{sec geometric motion} we prove Theorem \ref{thm blowing down limit}.  In Section \ref{sec representation} we give a representation formula for the blowing down limits. With these knowledge on blowing down limits, we prove Theorem \ref{thm minimal speed}  in Section \ref{sec minimal speed} by using the sliding method again.
%Finally, in Section \ref{sec discussion} we discuss several further problems,  which are most related to the conjecture of Hamel and %Nadirashvili posed  in \cite{Hamel2001entire} on the classification of entire solutions.

\section{Propagation phenomena}\label{sec propagation}
\setcounter{equation}{0}

%In this section we provide some observation about the propagation of  $v$ to $1$.

\subsection{Cone of monotonicity at large scales}
Standard parabolic regularity theory implies that $u$, $\nabla u$, $\nabla^2u$ and $\partial_tu$ are all bounded in $\R^n\times \R$.
 By the Lipschitz property of $u$ in $t$, $\sup_{x\in\R^n}u(x,t)$ is  a Lipschitz function of $t$.

 We start with the following simple observation, which is related to the hypothesis {\bf(H1)}.
\begin{prop}\label{prop dichotomy}
  Either $\sup_{x\in\R^n}u(x,t)\equiv 1$ or $\sup_{x\in\R^n}u(x,t)<1$ in  $(-\infty,+\infty)$.
\end{prop}
\begin{proof}
Denote
\[\mathcal{I}:=\left\{t: \sup_{x\in\R^n}u(x,t)=1\right\}.\]
By  continuity, $\mathcal{I}$ is a closed subset of $\R$.

We claim that $\mathcal{I}$ is also open. Therefore it is either empty or the entire real line. Indeed, if $\sup_{x\in\R^n}u(x,t_0)=1$, then there exist a sequence of points $x_j\in\R^n$ such that $u(x_j,t_0)\to1$. Let
\[u_j(x,t):=u(x_j+x,t_0+t).\]
By standard parabolic regularity theory and Arzela-Ascolli theorem, $u_j\to u_\infty$ in $C^{2,1}_{loc}(\R^n\times \R)$, where $u_\infty$ is an entire solution of \eqref{eqn}. Since $0\leq u_\infty\leq 1$ and $u_\infty(0,0)=1$, by {\bf(F1)} and the strong maximum principle, $u_\infty\equiv 1$. As a consequence, for any $\varepsilon>0$ and $t\in(-\varepsilon,\varepsilon)$,
\[\lim_{j\to\infty}u(x_j,t_0+t)=1.\]
Hence $\sup_{x\in\R^n}u(x,t)=1$ in $(t_0-\varepsilon,t_0+\varepsilon)$ and the claim follows.
\end{proof}

From now on it is always   assumed that ${\bf (H1)}$ holds, i.e. $\sup_{x\in\R^n}u(x,t)\equiv 1$ for any $t\in\R$.

\begin{lem}\label{lem close to 1}
For any $b>0$ and $R>0$, there exists a constant $\varepsilon:=\varepsilon(b,R)>0$ such that for any $(x,t)\in\R^n\times\R$, if $u(x,t)\geq 1-\varepsilon$, then $u\geq 1-b$ in $B_{R}(x)\times(t-R,t+R)$.
\end{lem}
\begin{proof}
 This follows from a contradiction argument  similar to the proof of Proposition \ref{prop dichotomy}, by applying the strong maximum principle to the limiting solution.
\end{proof}

The following result is essentially  \cite[Lemma 5.1]{Aronson1978multidimensional} (see also \cite[Lemma 3.5]{polacik2011threshold}).
We will use the notations of forward and backward light cones in space-time:
\[
   \left\{
\begin{aligned}
& \mathcal{C}^+_{\lambda}(x,t):=\left\{(y,s): ~~ s>t, |y-x|<\lambda(s-t)\right\},\\
& \mathcal{C}^-_{\lambda}(x,t):=\left\{(y,s): ~~ s<t, |y-x|<\lambda(t-s)\right\}.
\end{aligned}\right.
\]

\begin{lem}[Propagation to state $1$]\label{lem propagation to 1}
  There exists a constant $b_1\in(0,1)$  such that for any $b\in[0,b_1)$ and $\delta>0$, there exists an $R:=R(b,\delta)$ so that the following holds. If $w$ is the solution to the   Cauchy problem
  \begin{equation}\label{Cauchy problem}
    \left\{
\begin{aligned}
& \partial_t w-\Delta w=f(w)  \quad & \mbox{in } ~~ \R^n\times(0,+\infty),\\
&w(0)=\left(1-b\right)\chi_{B_R},
\end{aligned}\right.
  \end{equation}
  where $R\geq R(b,\delta)$, then
  \[w(x,t)> 1-b \quad \mbox{in } ~~ \mathcal{C}^+_{\kappa_\ast-\delta}(0,0).\]
\end{lem}
By decreasing $b_1$ further, we may assume $f^\prime\leq f^\prime(1)/2$ in $[1-b_1,1]$.

For applications below, we need an  a priori estimates for a linear parabolic equation. 
\begin{lem}\label{lem A.1}
Given a constant $M>0$, if $w$ satisfies
 \[
\left\{\begin{aligned}
&\partial_tw-\Delta w\leq -Mw  \quad &\mbox{in } B_1\times(-1,0),\\
&0\leq w\leq 1      \quad &\mbox{in } B_1\times(-1,0),
\end{aligned}\right.
\]
then
\[ w\leq \frac{C}{M}  \quad \mbox{in } B_{1/2}\times(-1/2,0).\]
\end{lem}
This  can be proved, for example, by constructing a suitable sup-solution.
The first application of this lemma is
\begin{lem}\label{lem inf u}
  For any entire solution $u$, if it is not exactly $1$, then
  \[\inf_{\R^n\times\R}u<1-b_1.\]
\end{lem}
\begin{proof}
  Assume by the contrary, $u\geq 1-b_1$ everywhere. By our choice of $b_1$, we get
  \begin{equation}\label{inequality near 1}
    \partial_t(1-u)-\Delta (1-u)\leq \frac{f^\prime(1)}{2}(1-u) \quad \mbox{in} ~~ \R^n\times\R.
  \end{equation}
  An iteration of Lemma \ref{lem A.1}  gives $u\equiv 1$.
\end{proof}

The next lemma is our main technical tool for the proof of Lipschitz property.
\begin{lem}\label{lem forward Lip}
There exist two constants $D>0$,  $0<b_2<b_1$ so that the following holds.
For any $(x,t)\in\{u=1-b_2\}$,
\[u>1-b_2 \quad  \mbox{in} \quad  \mathcal{C}^+_{\kappa_\ast-\delta}(x,t+D).\]
\end{lem}
\begin{proof}
Take $R:=R(b_1,\delta)$ according to Lemma \ref{lem propagation to 1}, $b_2:=\varepsilon(b_1,R)$ according to Lemma \ref{lem close to 1}. Then $u(x,t)=1-b_2$ implies   $u(y,t)\geq 1-b_1$ for any $y\in B_{R}(x)$. Combining Lemma \ref{lem propagation to 1} and comparison principle, we deduce that $u> 1-b_1$ in $\mathcal{C}_{\kappa_\ast-\delta}^+(x,t)$.

Now $1-u$ satisfies the differential inequality \eqref{inequality near 1} in $\mathcal{C}_{\kappa_\ast-\delta}^+(x,t)$.
By Lemma \ref{lem A.1}, we find a $D>0$, which depends only on $b_1,b_2$ and $f^\prime(1)$, such that $u>1-b_2$ in $\mathcal{C}^+_{\kappa_\ast-\delta}(x,t+D)$.
\end{proof}

Three corollaries follow from this lemma. The first two of them are rather direct consequences of this lemma.
\begin{coro}\label{coro backward cone}
  For any $(x,t)\in\{u=1-b_2\}$,
\[u<1-b_2 \quad  \mbox{in} \quad  \mathcal{C}^-_{\kappa_\ast-\delta}(x,t-D).\]
\end{coro}
\begin{coro}
  For any $(x,t)\in\{u=1-b_2\}$, $\{u=1-b_2\}$ lies between $\partial \mathcal{C}^-_{\kappa_\ast-\delta}(x,t-D)$ and $\partial \mathcal{C}^+_{\kappa_\ast-\delta}(x,t-D)$.
\end{coro}
\begin{coro}\label{coro not stationary}
  If $u$ is an entire solution of \eqref{eqn} satisfying {\bf(H1)}, then it cannot be independent of $t$, unless $u\equiv 1$.
\end{coro}
\begin{proof}
  Assume by the contrary, $\partial_tu\equiv 0$ in $\R^n\times \R$. By {\bf (H1)}, there exists  a point $(x,0)\in\{u=1-b_2\}$. By Lemma \ref{lem forward Lip}, $u>1-b_2$ in $\mathcal{C}_{\kappa_\ast-\delta}(x,D)$. Then we get $u\geq 1-b_2$ everywhere. By Lemma \ref{lem inf u}, $u\equiv 1$.
\end{proof}

\begin{prop}\label{prop construction of Lip graphs}
  The level set $\{u=1-b_2\}$ belongs to the $D$-neighborhood of a globally Lipschitz graph $\{t=h_\ast(x)\}$.
\end{prop}
\begin{proof}
Let
 \[h_\ast(x):=\inf_{(y,s)\in\{u=1-b_2\}} \left[s+ D+\frac{|x-y|}{\kappa_\ast-\delta}\right].\]
It is a globally Lipschitz function on $\R^n$, with its Lipschitz constant at most $(\kappa_\ast-\delta)^{-1}$. To check this, we need only to show that $h_\ast>-\infty$ at one point. (This then implies that it is finite everywhere.) In fact, take an arbitrary point $(x_0,t_0)\in\{u=1-b_2\}$. (The existence of such a point is guaranteed by {\bf (H1)} and Lemma \ref{lem inf u}.) By Corollary \ref{coro backward cone}, we see for any $ (y,s)\in\{u=1-b_2\}$,
\[|y-x_0|>\left(\kappa_\ast-\delta\right)\left(t_0-D-s\right).\]
In other words, $(x_0,t_0)\notin\mathcal{C}_{\kappa_0-\delta}^+(y,s+ D)$. Then by definition, we get
\[h_\ast(x_0)\geq t_0. \qedhere\]
\end{proof}

We modify $h_\ast$ into a smooth function. Take a  standard cut-off function $\eta\in C_0^\infty(\R^n)$, $\eta\geq 0$ and $\int_{\R^n}\eta=1$. Define
\[h^\ast(x):=\int_{\R^n}\eta(x-y)\left[h_\ast(y)+1\right]dy.\]
It is directly verified that $h^\ast$ has the same Lipschitz constant with $h_\ast$. Moreover, by choosing $\eta$ suitably, we have
\begin{equation}\label{2.1}
h_\ast \leq h^\ast\leq h_\ast+2.
\end{equation}
  Denote
\[\Omega^\ast:=\left\{(x,t):~~ t>h^\ast(x)\right\}.\]

\begin{lem}\label{lem comparison with linear eqn}
There exists a universal constant $c<1$ such that
\[ cb_2\leq 1-u\leq  b_2  \quad \mbox{on } \{t=h^\ast(x)\}.\]
\end{lem}
\begin{proof}
The second inequality is a direct consequence of the fact that $u>1-b_2$ in $\Omega^\ast$, thanks to \eqref{2.1} and the construction of $h_\ast$ in the proof of Proposition \ref{prop construction of Lip graphs}.

  The first inequality follows by applying Harnack inequality to the linear parabolic equation
  \[\partial_t\left(1-u\right)-\Delta\left(1-u\right)=V\left(1-u\right)\]
  in the parabolic cylinder $B_{3\sqrt{D}}(x)\times (t-9D,t+9D)$.
In the above
$V:=-f(u)/(1-u)$ is an $L^\infty$ function.
\end{proof}

\subsection{Proof of Theorem \ref{main result 3}}
Before proving Theorem \ref{main result 3}, first we need to construct a comparison function. Consider the  problem
\begin{equation}\label{linear eqn}
\left\{\begin{aligned}
& \partial_tw^\ast-\Delta w^\ast=f^\prime(1)w^\ast,  \quad & \mbox{in } \Omega^\ast,\\
&w^\ast=1 \quad & \mbox{on } \partial\Omega^\ast.
\end{aligned}\right.
\end{equation}
\begin{prop}\label{prop property for linear eqn}
\begin{enumerate}
  \item  There exists a unique solution of \eqref{linear eqn} in $L^\infty(\Omega^\ast)\cap C^\infty(\overline{\Omega^\ast})$.
  \item  There exists a universal constant $C$ such that for any $(x,t)\in \Omega^\ast$,
  \begin{equation}\label{exponential decay}
  \frac{1}{C}e^{-C\left[ t-h^\ast(x)\right]}\leq w^\ast(x,t)\leq Ce^{-\frac{ t-h^\ast(x)}{C}}.
\end{equation}
\item There exists a universal constant $C$ such that
  \begin{equation}\label{gradient  bound 1}
  \frac{|\nabla w^\ast|}{w^\ast}+\frac{|\partial_tw^\ast|}{w^\ast}\leq C \quad \mbox{in  }~~ \Omega^\ast.
\end{equation}
\item There exists a universal constant $c$ such that
\begin{equation}\label{time derivative bound 1}
  \frac{\partial_tw^\ast}{w^\ast}\leq -c \quad \mbox{in  }~~  \Omega^\ast.
\end{equation}

\item
There exists a universal constant $C$ such that
  \begin{equation}\label{Lip}
   \frac{|\nabla w^\ast|}{|\partial_tw^\ast|}\leq C \quad \mbox{in  }~~  \Omega^\ast.
\end{equation}
As a consequence, all level sets of $w^\ast$ are Lipschitz graphs in the $t$-direction.
\end{enumerate}
\end{prop}
\begin{proof}
(i) Existence, uniqueness and regularity of the solution can be proved as in Ole\u{\i}nik and Radkevi\v{c} \cite[Chapter 1]{Oleinik1969book}, see respectively  Section 5, Section 6  and   Section 8 therein.

(ii) The exponential upper bound follows from iteratively applying the estimate in Lemma \ref{lem A.1}. The lower bound follows from  iteratively applying the standard Harnack inequality.

(iii) By the regularity theory in \cite{Oleinik1969book}, there exists a universal constant $C$ such that
\[|\nabla w^\ast|+|\partial_t w^\ast|\leq C \quad \mbox{in }  \Omega^\ast.\]
Hence for any constant vector $(\xi,s)\in\R^{n+1}$, $\xi\cdot\nabla w^\ast+s\partial_tw^\ast$ is  a bounded solution of \eqref{linear eqn}. As in (ii), it converges to $0$ as $t-h^\ast(x)\to+\infty$.  Then \eqref{gradient  bound 1} follows from an application of the comparison principle.

(iv) For any $\rho>0$, in the half ball
 \[\mathcal{B}_{\rho}(0,0):=\left\{(x,t): \quad |x|^2+t^2<\rho^2, ~~ t<0 \right\},\]
 the function
 \[w_\rho(x,t):= e^{\alpha\left(|x|^2+t^2-\rho^2\right)}\]
 is a sup-solution of \eqref{linear eqn}, provided that $\alpha$ is small enough (depending only on $\rho$, the space dimension $n$ and $f^\prime(1)$).

On $\left\{|x|^2+t^2=\rho^2, ~~ t<-\left(\kappa_\ast-\delta\right)^{-1}|x|\right\}$, there exists a  constant $c(\rho)>0$ such that
\[\partial_tw_\rho(x,t)=2\alpha tw_\rho(x,t)\leq -c(\rho).\]
Since $\sup_{\R^n}|\nabla^2h^\ast|\leq C$, for any $(x,h^\ast(x))$, there exists an half ball $\mathcal{B}_{1/C}(y,s)$ tangent to $\partial \Omega^\ast$ at this point. Moreover, because $|\nabla h^\ast(x)|\leq \left(\kappa_\ast-\delta\right)^{-1}$,
\[ t-s\leq -\left(\kappa_\ast-\delta\right)^{-1}|x-y|.\]
By the comparison principle in  $\mathcal{B}_{1/C}(y,s)$, $w_{1/C}(\cdot-(y,s))\geq w^\ast$ in $\mathcal{B}_{1/C}(y,s)$. Therefore
\[\partial_tw^\ast\left(x,h^\ast(x)\right)\leq -c.\]
Then \eqref{time derivative bound 1} follows from an application of the comparison principle as in (iii).

(v) This is a direct consequence of \eqref{gradient  bound 1} and \eqref{time derivative bound 1}.
\end{proof}

\begin{lem}\label{lem comparison with linear eqn 2}
  There exists a universal constant $C$ such that
  \begin{equation}\label{control from above and below}
     \frac{b_2}{C}\leq \frac{1-u}{w^\ast}\leq Cb_2 \quad \mbox{in }  \Omega^\ast.
  \end{equation}
\end{lem}
\begin{proof}
  For each $k\geq 1$, let
  \[\Omega^\ast_k:=\left\{(x,t): ~~ k-1<t-h^\ast(x)<k \right\}.\]

 Similar to \eqref{exponential decay}, for any $(x,t)\in \Omega^\ast$ we have
  \begin{equation}\label{exponential decay for u}
  \frac{1}{C}e^{-C\left[ t-h^\ast(x)\right]}\leq 1-u(x,t)\leq Ce^{-\frac{ t-h^\ast(x)}{C}}.
  \end{equation}
Hence there exists a $\sigma\in(0,1)$ such that
\[ \sigma^k f^\prime(1)\left(1-u\right)\leq \left[\partial_t-\Delta-f^\prime(1)\right]\left(1-u\right)\leq -\sigma^k f^\prime(1)\left(1-u\right) \quad \mbox{in } \Omega^\ast_k.\]

For each $k$, define $w_k^\ast$ inductively as the unique solution of
\begin{equation}\label{inductive eqn}
   \left\{\begin{aligned}
&\partial_tw_k^\ast-\Delta w_k^\ast=f^\prime(1)\left(1-\sigma^k\right)w_k^\ast  \quad &\mbox{in } \Omega^\ast_k,\\
&w_k^\ast=w_{k-1}^\ast     \quad &\mbox{on } \{t=h^\ast(x)+k-1\}.
\end{aligned}\right.
\end{equation}
Here $w_{0}^\ast:\equiv b_2$. As before, the existence and uniqueness of $w_k^\ast$ follows from  Ole\u{\i}nik and  Radkevi\v{c} \cite[Chapter 1]{Oleinik1969book}.

By inductively applying the comparison principle, we get
\begin{equation}\label{control from above 1}
  1-u\leq w_k^\ast \quad \mbox{in } \Omega^\ast_k.
\end{equation}

Next for each $k$, denote
\[M_k:=\sup_{\Omega^\ast_k}\frac{w_k^\ast}{w^\ast}.\]
A direct calculation shows that $\left(M_k w^\ast\right)^{1-\sigma^k}$ is a sup-solution of \eqref{inductive eqn} in $\Omega^\ast_{k+1}$. (Here we also need an inductive assumption that $M_k w^\ast <1$ on $\{t=h^\ast(x)+k\}$.) From this we deduce that
\[M_{k+1}\leq M_k^{1-\sigma^{k}}\sup_{\Omega^\ast_{k+1}}\left(w^\ast\right)^{-\sigma^{k}}\leq M_k^{1-\sigma^{k}} e^{Ck\sigma^k},\]
where the last inequality follows by substituting \eqref{exponential decay} to estimate $\inf_{\Omega^\ast_{k+1}}w^\ast$.

From this inequality it is readily deduced that
there exists a universal, finite upper bound on $M_k$ as $k\to\infty$. Combining this fact with \eqref{control from above 1} we obtain the upper bound in \eqref{control from above and below}.

The lower bound in \eqref{control from above and below} follows in the same way by considering
\begin{equation*}%\label{inductive eqn 2}
   \left\{\begin{aligned}
&\partial_tw_{k,\ast}-\Delta w_{k,\ast}=f^\prime(1)\left(1+\sigma^k\right)w_{k,\ast}  \quad &\mbox{in } \Omega^\ast_k,\\
&w_{k,\ast} =w_{k-1,\ast}     \quad &\mbox{on } \{t=h^\ast(x)+k-1\}.
\end{aligned}\right. \qedhere
\end{equation*}
\end{proof}

\begin{coro}\label{coro differential Harnack}
 There exists a universal constant $C>0$ such that
 \begin{equation}\label{gradient  bound 2}
    \frac{|\nabla u|+|\partial_tu|}{1-u}\leq C \quad \mbox{ in } ~~ \Omega^\ast.
 \end{equation}
\end{coro}
\begin{proof}
  For any $(x,t)\in \Omega^\ast$, standard gradient estimate gives
\begin{equation}\label{gradient estimate}
  |\nabla u(x,t)|+|\partial_tu(x,t)|\leq C\sup_{B_1(x)\times (t-1,t)}(1-u).
\end{equation}
By the previous lemma, for any $(y,s)\in B_1(x)\times (t-1,t)$,
\begin{equation}\label{Harnack}
  1-u(y,s)\leq Cw^\ast(y,s) \leq C^2w^\ast(x,t)\leq C^3 \left[1-u(x,t)\right].
\end{equation}
Here the second inequality follows by integrating \eqref{gradient  bound 1} along the segment from $(y,s)$ to $(x,t)$.

Substituting \eqref{Harnack} into \eqref{gradient estimate} we get \eqref{gradient  bound 2}.
\end{proof}

\begin{prop}\label{prop monotonicity in time 1}
  There exists a $b_0\in(0,b_2)$ such that in $\{u>1-b_0\}$,
  \begin{equation}\label{time derivative bound 2}
    \frac{\partial_tu}{1-u}\geq  c.
  \end{equation}
\end{prop}
\begin{proof}
  Assume by the contrary, there exists a sequence of points  $(x_k,t_k)$ such that  $u(x_k,t_k)\to 1$ but
  \begin{equation}\label{absurd assumption 1}
     \frac{\partial_tu(x_k,t_k)}{1-u(x_k,t_k)}\to 0.
  \end{equation}
Denote $R_k:=\mbox{dist}((x_k,t_k),\partial\Omega^\ast)$. Because $u_(x_k,t_k)\to1$, by \eqref{exponential decay for u}, $R_k$ goes to infinity as $k\to\infty$.

Let
\[u_k(x,t):=\frac{1-u(x_k+x,t_k+t)}{1-u(x_k,t_k)}, \quad w_k(x,t):=\frac{w^\ast(x_k+x,t_k+t)}{1-u(x_k,t_k)}.\]
By definition, $u_k(0,0)=1$, while by \eqref{control from above and below}, we have
\[ \frac{b_2}{C}\leq \frac{u_k}{w_k}\leq Cb_2 \quad \mbox{in } B_{cR_k}(0)\times(-cR_k,cR_k).\]
Furthermore, \eqref{gradient  bound 1} and \eqref{gradient  bound 2} are transformed into
\[\frac{|\nabla u_k|+|\partial_tu_k|}{u_k}\leq C, \quad \frac{|\nabla w_k|+|\partial_tw_k|}{w_k}\leq C \quad \mbox{in } B_{cR_k}(0)\times(-cR_k,cR_k).\]
Then by standard parabolic regularity theory, $u_k$ and $w_k$ are uniformly bounded in $C^{2+\alpha,1+\alpha/2}_{loc}(\R^n\times\R)$. After passing to a subsequence, $u_k\to u_\infty$, $w_k\to w_\infty$ in $C^{2,1}_{loc}(\R^n\times\R)$. Both of them are solutions of
\[\partial_t w-\Delta w=f^\prime(1) w \quad \mbox{in }~~ \R^n\times \R.\]
By \cite{Widder1963heat} or \cite{Lin2019ancient}, there exists a Borel measure supported on $\{\lambda=|\xi|^2\}\subset\R^{n+1}$ such that
\[ w_\infty(x,t)=\int_{\{\lambda=|\xi|^2\}}e^{\left[f^\prime(1)+\lambda\right]t+\xi\cdot x}d\mu(\xi,\lambda).\]

Because
\begin{equation}\label{2.2}
   \frac{b_2}{C}\leq \frac{u_\infty}{w_\infty}\leq Cb_2 \quad  \mbox{in }~~ \R^n\times \R,
\end{equation}
there exists a function $\Theta$ on $\{\lambda=|\xi|^2\}$ with $b_2/C\leq \Theta \leq Cb_2$ such that
\[ u_\infty(x,t)=\int_{\{\lambda=|\xi|^2\}}e^{\left[f^\prime(1)+\lambda\right]t+\xi\cdot x}\Theta(\xi,\lambda)d\mu(\xi,\lambda).\]
This follows by writing $u_\infty$ as the same integral representation with another measure $\widetilde{\mu}$, applying Radon-Nikodym theorem to these two measures, and then use \eqref{2.2} to estimate the differential $\frac{d\widetilde{\mu}}{d\mu}$.

Because $w_\infty$ still satisfies the inequality \eqref{time derivative bound 1}, the support of $\mu$ is contained in $\{\lambda\leq -f^\prime(1)-c\}$. Hence we also have
\begin{eqnarray*}
  \partial_tu_\infty&=& \int_{\{\lambda=|\xi|^2\}}\left[f^\prime(1)+\lambda\right]e^{\left[f^\prime(1)+\lambda\right]t+\xi\cdot x}\Theta(\xi,\lambda)d\mu(\xi,\lambda)  \\
    &\leq & -cu_\infty.
\end{eqnarray*}
This is a contradiction with \eqref{absurd assumption 1}.
\end{proof}

Theorem \ref{main result 3} follows by combining \eqref{gradient  bound 2} and \eqref{time derivative bound 2}.

\section{Proof of Theorem \ref{main result 4}}\label{sec case near 0}
\setcounter{equation}{0}

For simplicity, denote $h(x):=h_{1-b_0}(x)$.

\subsection{The combustion and bistable case}\label{sec bistable}
In these two cases we need the assumption {\bf (H2)}, that is,  $u(x,t)\to0$ uniformly as $\mbox{dist}((x,t),\{t=h(x)\})\to +\infty$.

First we use the sliding method to prove
\begin{prop}\label{prop monotonicity in bistable case}
  $u$ is increasing in $t$.
\end{prop}
\begin{proof}
  The fact that $\partial_tu>0$ in $\{u>1-b_0\}$ has been established in Proposition \ref{prop monotonicity in time 1}. Now we use the sliding method to show the remaining case.

For any $\lambda\in\R$, let
\[u^\lambda(x,t):=u(x,t+\lambda).\]
We want to  show that for any $\lambda>0$, $u^\lambda\geq u$ in $\R^n\times\R$.

{\bf Step 1.} If $\lambda$ is large enough, $u^\lambda\geq u$ in $\R^n\times\R$.

By {\bf(H2)}, there exists a constant $L>0$ such that
\begin{equation}\label{small in the right}
  \sup_{\{t<h(x)-L\}}u \ll 1.
\end{equation}
 If $\lambda> L$, we have
\[ u^\lambda \geq 1-b_0 \geq u \quad \mbox{in } \{h(x)-L\leq t \leq h(x)\}.\]
In $\{t<h(x)-L\}$, we have
\[\partial_t\left(u-u^\lambda\right)_+-\Delta \left(u-u^\lambda\right)_+\leq V\left(u-u^\lambda\right)_+,\]
where
\[ V:=
\begin{cases}
  \frac{f(u)-f(u^\lambda)}{u-u^\lambda}, & \mbox{if }  u>u^\lambda\\
  f^\prime(u), & \mbox{otherwise}.
\end{cases}
\]
By \eqref{small in the right} and {\bf(F4)}, $V\leq 0$ in $\{t<h(x)-L\}$. Because $\left(u-u^\lambda\right)_+=0$ on $\{t=h(x)-L\}$ and
$\left(u-u^\lambda\right)_+\to 0$ as $\mbox{dist}((x,t),\{t=h(x)-L\})\to+\infty$, by the maximum principle we obtain
\[u\leq u^\lambda  \quad \mbox{in } ~~ \{t<h(x)-L\}.\]

{\bf Step 2.} Now
\[ \lambda_\ast:=\inf\left\{\lambda: ~~ \forall ~ \lambda^\prime>\lambda, ~~ u^{\lambda^\prime}\geq u ~~ \mbox{ in } \R^n\times\R\right\}\]
is well defined. We claim that $\lambda_\ast=0$.

By continuity, $u^{\lambda_\ast}\geq u$ in $\R^n\times\R$. By the strong maximum principle, either $u^{\lambda_\ast}>u$ strictly or $u^{\lambda_\ast}\equiv u$. The later is excluded if $\lambda_\ast>0$, because in this case $u^{\lambda_\ast}>u$ in $\{t>h(x)\}$.

{\bf Claim.} If $\lambda_\ast>0$, for any $L>0$, there exists a constant $\varepsilon_1:=\varepsilon(\lambda_\ast,L)>0$ such that
\[
  u^{\lambda_\ast}-u\geq \varepsilon_1  \quad \mbox{in }~~ \{h(x)-L\leq t \leq h(x)\}.
\]
%Indeed, in $\{t\geq h_{1-b_3}(x)\}$, this is a direct consequence of  Lemma \ref{prop monotonicity in time 1}.
We prove this claim by contradiction. Assume  there exists a sequence of points $(x_i,t_i)\in \{h(x)-L\leq t \leq h(x)\}$ such that $u^{\lambda_\ast}(x_i,t_i)-u(x_i,t_i)\to 0$. Set
\[u_i(x,t):=u(x_i+x,t_i+t).\]
They satisfy the following conditions:
\begin{itemize}
\item there exists a constant $b(L)\in(0,1)$ such that $b(L)\leq u_i(0,0)\leq 1-b(L)$;
\item $u_i^{\lambda_\ast}\geq u_i$ in $\R^n\times\R$;
  \item $u_i^{\lambda_\ast}-u_i\geq c\lambda_\ast \left(1-u_i\right)$ in $\{u_i\geq 1-b_0\}$ (thanks to Proposition \ref{prop monotonicity in time 1});
  \item $u_i^{\lambda_\ast}(0,0)-u_i(0,0)\to 0$.
\end{itemize}
These lead to a contradiction after letting $i\to+\infty$, and the proof of this claim is complete.

By this claim and Proposition \ref{prop monotonicity in time 1}, for any $L>0$, we find another constant $\varepsilon_2:=\varepsilon_2(\lambda_\ast,L)>0$ such that, if $\lambda\geq \lambda_\ast-\varepsilon_2$,
\begin{equation}\label{positive lower bound}
  u^{\lambda} \geq u  \quad \mbox{in }~~ \{t\geq h(x)-L\}.
\end{equation}

%Because $f^\prime(0)<0$, by {\bf(H2)}, after enlarging $L$, we may assume $f^\prime(s)<0$ for any $s<\max\{\sup_{ \{t< h(x)-L\}}u^\lambda,\sup_{ \{t< h(x)-L\}}u\}$. Then we have
%\[ \partial_t\left(u^\lambda-u\right)-\Delta\left(u^\lambda-u\right)=V\left(u^\lambda-u\right) \quad \mbox{in} ~~ \{t<h(x)-L\},\]
%where
%\[ V:=
%\begin{cases}
 % \frac{f\left(u^\lambda\right)- f(u)}{u^\lambda-u} , & \mbox{if } u^\lambda\neq u,\\
 % f^\prime(u), & \mbox{otherwise},
%\end{cases}
%\]
%is non-positive.
Then as in Step 1, by \eqref{positive lower bound} and the comparison principle,  for these $\lambda$, $u^\lambda\geq u$ in $\{t<h(x)-L\}$ (hence everywhere in $\R^n\times\R$).
This is a contradiction with the definition of $\lambda_\ast$. Therefore we must have $\lambda_\ast=0$.
\end{proof}

Combining this proposition with Corollary \ref{coro not stationary} and strong maximum principle, we obtain
\begin{coro}\label{coro strict monotonicity}
  In $\R^n\times\R$, $\partial_tu>0$ strictly.
\end{coro}

\begin{prop}\label{prop 3.3}
  There exists a universal constant $c>0$ such that
  \[ \partial_tu\geq c|\nabla u| \quad \mbox{in }~~ \R^n\times\R.\]
\end{prop}
\begin{proof}
In $\{u>1-b_0\}$, this inequality follows by combining Corollary \ref{coro differential Harnack} and Proposition \ref{prop monotonicity in time 1}.

In view of Corollary \ref{coro strict monotonicity},  following the argument in the second step of the proof of Proposition \ref{prop monotonicity in bistable case}, for any $L>0$, we find a positive lower bound for $\partial_tu$ in $\{h(x)-L\leq t \leq h(x)\}$. Hence trivially we have
\[|\nabla u|\leq C \leq \frac{C}{c}\partial_tu \quad \mbox{in }~~ \{h(x)-L\leq t \leq h(x)\}.\]

Finally, we apply the maximum principle to the linearized equation
  \[\left(\partial_t-\Delta\right)\left(\partial_tu-c\xi\cdot\nabla u\right)=f^\prime(u)\left(\partial_tu-c\xi\cdot\nabla u\right)\]
to show that $\partial_tu-c\xi\cdot\nabla u\geq 0$ in $\{t<h(x)-L\}$, where $\xi$ is an arbitrary unit vector in $\R^n$ and $c>0$ is a small constant.
\end{proof}

Theorem \ref{main result 4} is a direct consequence of this proposition.

\subsection{A remark on the monostable case}\label{subsec monostable case}
In this subsection we give a remark on the monostable case.

For this case, we note the following important   ``hair trigger" phenomena (see \cite[Theorem 3.1]{Aronson1978multidimensional}).
\begin{lem}\label{lem propagation to 1 in monostable case}
For any $\lambda\in(0,1)$, $\delta>0$ and $(x,t)\in\R^n\times\R$, there exists a  constant $D:=D(x,t,\lambda)>0$  such that
  \[u> \lambda \quad \mbox{in } ~~ \mathcal{C}^+_{\kappa_\ast-\delta}(x,t+D).\]
\end{lem}

\begin{lem}\label{lem 3.1}
  If $f$ is   monostable, then $u\to0$ uniformly as $\mbox{dist}((x,t),\{u\geq 1-b_0\})\to +\infty$.
\end{lem}
\begin{proof}
 This follows from the following Liouville type result: suppose $u$ is an entire solution of \eqref{eqn} satisfying $0\leq u \leq 1-b_0$, then $u\equiv 0$. This Liouville theorem is a direct consequence of Lemma \ref{lem propagation to 1 in monostable case}.
\end{proof}

Unfortunately, in this case we need to assume the following assumption:
  \begin{equation}\label{differential Harnack 2}
    \frac{|\partial_tu|+|\nabla u|}{u} \leq C  \quad \mbox{in }~~ \{t<h(x)\}.
  \end{equation}
Of course, if $u$ is a travelling wave solution, this assumption holds by applying  standard elliptic Harnack inequality and interior gradient estimates (see Lemma \ref{lem gradient bound 2} below), but we do not know how to prove the parabolic case.

\begin{lem}\label{lem 3.2}
Given $\kappa>0$, assume $u$ is an entire positive solution of
\[\partial_tu-\Delta u=\kappa u.\]
Then
\[\frac{|\nabla u|}{\partial_t u}\leq \frac{1}{2\sqrt{\kappa}} \quad \mbox{in }~~ \R^n\times\R.\]
As a consequence, all level sets of $u$ are Lipschitz graphs in the $t$ direction, with their Lipschitz constants at most $2\sqrt{\kappa}$.
\end{lem}
\begin{proof}
  By \cite{Widder1963heat} or \cite{Lin2019ancient}, there exists a Borel measure $\mu$ supported on $\{\lambda=|\xi|^2\}\subset\R^{n+1}$ such that
\[ u(x,t)=\int_{\{\lambda=|\xi|^2\}}e^{\left[\kappa+\lambda\right]t+\xi\cdot x}d\mu(\xi,\lambda).\]
Then we have
\begin{eqnarray*}
  \partial_tu(x,t) &=& \int_{\{\lambda=|\xi|^2\}}\left[\kappa+\lambda\right]e^{\left[\kappa+\lambda\right]t+\xi\cdot x}d\mu(\xi,\lambda) \\
   &=& \int_{\{\lambda=|\xi|^2\}}\left[\kappa+|\xi|^2\right]e^{\left[\kappa+\lambda\right]t+\xi\cdot x}d\mu(\xi,\lambda) \\
   &\geq&2\sqrt{\kappa}\int_{\{\lambda=|\xi|^2\}}|\xi|e^{\left[\kappa+\lambda\right]t+\xi\cdot x}d\mu(\xi,\lambda)\\
  &\geq&  2\sqrt{\kappa}|\nabla u(x,t)|. \qedhere
\end{eqnarray*}
\end{proof}

\begin{coro}
  There exists a constant $L>0$ such that
\[\frac{|\nabla u|}{\partial_tu}\leq \frac{1}{4\sqrt{f^\prime(0)}} \quad \mbox{in } ~~ \{t<h(x)-L\}.\]
\end{coro}
\begin{proof}
For any $(x_i,t_i)\in\{t<h(x)\}$ with $t_i-h(x_i)\to -\infty$, by Lemma \ref{lem 3.1}, $u(x_i,t_i)\to 0$. Set
\[u_i(x,t):=\frac{u(x_i+x,t_i+t)}{u(x_i,t_i)}.\]
By definition, $u_i>0$ and $u_i(0,0)=1$. Integrating \eqref{differential Harnack 2}, we see $u_i$ are uniformly bounded in any compact set of $\R^n\times\R$. Then by standard parabolic regularity theory, we can take a subsequence $u_i\to u_\infty$ in $C^{2,1}_{loc}(\R^n\times\R)$. Here $u_\infty$ is an entire solution of
\[\partial_tu_\infty-\Delta u_\infty =f^\prime(0) u_\infty.\]
The claim then follows from Lemma \ref{lem 3.2}.
\end{proof}

Theorem \ref{main result 4} in the monostable case (under the hypothesis \eqref{differential Harnack 2}) follows from the same sliding method as in the previous subsection.

\section{An elliptic Harnack inequality}\label{sec elliptic Harnack}
\setcounter{equation}{0}
From now on, unless otherwise stated, it is always assumed that ${\bf (F1-F4)}$ and ${\bf (H1-H2)}$ hold.
In this section we prove an elliptic Harnack inequality for $u$. This will be used in the blowing down analysis in the next section.

In $\{t>h(x)\}$, what we want has been given in Corollary \ref{coro differential Harnack}, so here we consider the other part $\{t<h(x)\}$.
\begin{prop}\label{prop differential Harnack 1}
  There exists a universal constant $C>0$ such that
  \begin{equation}\label{differential Harnack 1}
    \frac{|\partial_tu|+|\nabla u|}{u} \leq C  \quad \mbox{in }~~ \left\{t<h(x)\right\}.
  \end{equation}
\end{prop}

Before proving this proposition, we first notice the following exponential decay   of $u$ in $\{t<h(x)\}$ for later use.
\begin{prop}\label{prop negative part}
Under the hypothesis {\bf(H2)},
\begin{equation}\label{exponential decay 2}
  u(x,t)\leq Ce^{c\left[t-h(x)\right]} \quad \mbox{in} \quad \{t<h(x)\}.
\end{equation}
\end{prop}
\begin{proof}
 Because  $f$ is of combustion type or bistable, by choosing $L$ large enough, we have
 \[ \partial_tu-\Delta u\leq 0  \quad \mbox{in} \quad \{t<h(x)-L\}.\]
Take a radially symmetric function $\varphi\in C^2(\R^n)$ such that $\varphi\leq 0$, $\varphi(x)\equiv -2\kappa_\ast|x|$ outside a large ball, $|\Delta\varphi|\ll 1$ and $|\nabla\varphi|\leq C$ in $\R^n$. By taking a small $\mu>0$, the function
\[w(x,t):=e^{\mu\left[t-\varphi(x)\right]}\]
is a super-solution of the heat equation in $\mathcal{D}:=\{(y,t): t<\varphi(x)\}$. Moreover, $w=1$ on $\partial\mathcal{D}$.

For each $(x,t)\in \{t<h(x)\}$, by enlarging $L$ further (depending on the Lipschitz constant of $h$, but independent of $x$), the domain
\[\mathcal{D}_x:=\left\{(y,s): s<h(x)-2L-\varphi(y-x)\right\}\subset \{s<h(y)-L\}.\]
A comparison with a suitable translation of $w$ leads to \eqref{exponential decay 2}.
\end{proof}

Take a large $L>0$ so that $u\ll 1$ in $\left\{t<h(x)-L\right\}$. This is possible by {\bf (H2)}.
In $\left\{h(x)-L\leq t \leq h(x)\right\}$, \eqref{differential Harnack 1} is a direct consequence of the facts that $u$ has a positive lower bound here while both $\partial_tu$ and $|\nabla u|$ are bounded. It thus remains  to show that \eqref{differential Harnack 1} holds in $\{t<h(x)-L\}$.

We first prove the combustion case.
\begin{proof}[Proof of Proposition \ref{prop differential Harnack 1} in the combustion case]

If $f$ is of combustion type,  $u$, $\partial_tu$ and $\partial_{x_i}u$ all satisfy the heat equation in $\{t<h(x)-L\}$. The estimate \eqref{differential Harnack 1} then follows from the comparison principle. For example, because both $\partial_tu$ and $u$ converge to $0$ uniformly as $\mbox{dist}((x,t),\{t=h(x)\})\to +\infty$, if $\partial_tu\leq Mu$ on $\{t=h(x)-L\}$ for some constant $M>0$, then
\[\partial_tu\leq Mu  \quad \mbox{in } ~~ \{t<h(x)-L\}. \qedhere\]
\end{proof}

Next, we prove the bistable case.
\begin{proof}[Proof of Proposition \ref{prop differential Harnack 1} in the bistable case]

Take a $b\in(0,1)$ sufficiently small so that $f\in C^{1,\alpha}([0,b])$,
\begin{equation}\label{4.1}
  f^\prime(u)\leq f^\prime(0)/2 \quad  \mbox{and} \quad |f^\prime(u)-f^\prime(0)|\leq Cu^\alpha, \quad \mbox{for any}~~~ u\in[0,b].
\end{equation}

For any $\lambda\in(0,b)$, denote $\Omega_\lambda:=\{u<\lambda\}$. A direct calculation using \eqref{4.1} shows that for some universal constant $C>1$ (independent of $\lambda$),
\[\partial_tu^{1-C\lambda^\alpha}-\Delta u^{1-C\lambda^\alpha}\geq f^\prime(u) u^{1-C\lambda^\alpha}.\]
On the other hand, $\partial_tu$ is a solution of this linearized equation.

Therefore, if we denote
\[ M(\lambda):= \sup_{\partial\Omega_\lambda}\frac{\partial_tu}{u^{1-C\lambda^\alpha}}=\sup_{\partial\Omega_\lambda}\frac{\partial_tu}{\lambda^{1-C\lambda^\alpha}},\]
applying the comparison principle as in the proof of the combustion case, we obtain
\begin{equation}\label{comparison in backward domain}
  \partial_tu\leq M(\lambda)u^{1-C\lambda^\alpha} \quad \mbox{in} \quad \Omega_\lambda.
\end{equation}
From this inequality and the fact that $\partial\Omega_{\lambda/2}\subset \Omega_\lambda$, we deduce that
\[ M\left(\frac{\lambda}{2}\right)\leq M(\lambda)  \left(\frac{\lambda}{2}\right)^{-C\left(1-2^{-\alpha}\right) \lambda^\alpha}.\]
This inequality implies that
\[\limsup_{\lambda\to0}M(\lambda)<+\infty.\]
Substituting this estimate into \eqref{comparison in backward domain}, we find a constant $C$ such that for any $\lambda\in(0,b)$, in $\Omega_\lambda\setminus \Omega_{\lambda/2}$,
\begin{equation}\label{estimate for u_t}
  \partial_tu  \leq   Cu^{1-C\left(2u\right)^\alpha} \leq  2Cu.
\end{equation}
  Here to deduce  the last inequality, we have used the inequality (perhaps after choosing a smaller $b$)
  \[ u^{-C2^\alpha u^\alpha}\leq 2, \quad \mbox{if}~~~ u\leq b.\]

Finally, the estimate for $|\nabla u|/u$ follows by combining \eqref{estimate for u_t} and Proposition \ref{prop 3.3}.
\end{proof}

\section{Blowing down analysis}\label{sec blowing down}
\setcounter{equation}{0}

Recall that the one dimensional travelling wave $g$ (see \eqref{1D wave}) is strictly increasing, and it converges to $1$ and $0$ exponentially as $t\to \pm\infty$. In fact, by {\bf(F1)} and ${\bf (F4)}$, there exist four positive constants $\alpha_\pm$ and $\beta_\pm$ such that
\[g(t)=1-\alpha_+e^{-\beta_+t}+O\left(e^{-(1+\alpha)\beta_+t}\right) \quad \mbox{as }~~ t\to+\infty,\]
\[g(t)= \alpha_-e^{\beta_-t}+O\left(e^{(1+\alpha)\beta_-t}\right) \quad \mbox{as }~~ t\to-\infty,\]
where
\[\beta_+:=-\lim_{t\to+\infty}\frac{g^{\prime\prime}(t)}{g^\prime(t)}=\frac{-\kappa_\ast+\sqrt{\kappa_\ast^2-4f^\prime(1)}}{2},\]
\[\beta_-:=\lim_{t\to-\infty}\frac{ g^{\prime\prime}(t)}{g^\prime(t)}=\frac{\kappa_\ast+\sqrt{\kappa_\ast^2-4f^\prime(0)}}{2}.\]
Because $f^\prime(0)\leq 0$, $\beta_-\geq\kappa_\ast$.

Following \cite{Barles1992front}, set $\Phi:=g^{-1}\circ u$. It satisfies
\begin{equation}\label{distance eqn 2}
  \partial_t\Phi-\Delta\Phi=\kappa_\ast+\frac{g^{\prime\prime}(\Phi)}{g^\prime(\Phi)}\left(|\nabla\Phi|^2-1\right).
\end{equation}

\begin{lem}\label{lem Lipschitz}
  There exists a universal constant $C>0$ such that
  \[ |\partial_t\Phi|+|\nabla\Phi|\leq C \quad  \mbox{in } ~~ \R^n\times\R.\]
\end{lem}
\begin{proof}
  By Proposition \ref{prop differential Harnack 1},
\[|\partial_t\Phi|\leq C\frac{|\partial_t u|}{u}\leq C, \quad |\nabla\Phi|\leq C\frac{|\nabla u|}{u}\leq C, \quad \mbox{in} ~~\{u\leq 1-b_0\}.\]
By Corollary \ref{coro differential Harnack},
 \[|\partial_t\Phi|\leq C\frac{|\partial_t u|}{1-u}\leq C, \quad |\nabla\Phi|\leq C\frac{|\nabla u|}{1-u}\leq C, \quad \mbox{in}~~ \{u\geq 1-b_0\}.\qedhere\]
\end{proof}

\begin{lem}[Semi-concavity]\label{lem semi-concavity}
There exists a universal constant $C$ such that for any $(x,t)\in\{\Phi>0\}$,
\[\nabla^2\Phi(x,t)\leq \frac{C}{\Phi(x,t)},\]
and for any $(x,t)\in\{\Phi<0\}$,
\[\nabla^2\Phi(x,t)\geq \frac{C}{\Phi(x,t)}.\]
\end{lem}
This follows from a standard maximum principle argument applied to $\eta\nabla^2\Phi(\xi,\xi)$, for any $\xi\in\R^n$ and a suitable cut-off function $\eta$.

For each $\varepsilon>0$, let $\Phi_\varepsilon(x,t):=\varepsilon\Phi(\varepsilon^{-1}x,\varepsilon^{-1}t)$. It satisfies
\begin{equation}\label{vanishing viscosity eqn 1}
  \partial_t\Phi_\varepsilon-\varepsilon\Delta\Phi_\varepsilon =\kappa_\ast
  +\frac{g^{\prime\prime}(\varepsilon^{-1}\Phi_\varepsilon)}{g^\prime(\varepsilon^{-1}\Phi_\varepsilon)}\left(|\nabla\Phi_\varepsilon|^2-1\right).
\end{equation}
By the uniform Lipschitz bound on $\Phi_\varepsilon$ from Lemma \ref{lem Lipschitz}, there exists a subsequence of $\varepsilon\to0$ such that
$\Phi_\varepsilon\to\Phi_\infty$ in $C_{loc}(\R^n\times\R)$.

The limit $\Phi_\infty$ may depend on the choice of subsequences. But for notational simplicity, we will always write $\varepsilon\to0$ instead of $\varepsilon_i\to0$.

By standard vanishing viscosity method, we get
\begin{lem}\label{lem blowing down limit}
  In the open set $\{\Phi_\infty>0\}$, $\Phi_\infty$ is a viscosity solution of
  \begin{equation}\label{H-J eqn 3}
    \partial_t\Phi_\infty+\beta_+|\nabla\Phi_\infty|^2-\kappa_\ast-\beta_+=0.
  \end{equation}

  In the open set $\{\Phi_\infty<0\}$ (if non-empty), $\Phi_\infty$ is a viscosity solution of
  \begin{equation}\label{H-J eqn 4}
    \partial_t\Phi_\infty-\beta_-|\nabla\Phi_\infty|^2-\kappa_\ast+\beta_-=0.
  \end{equation}
\end{lem}

Since $h(x)$ is globally Lipschitz on $\R^n$, by taking a further subsequence, we may also assume
 \[ \varepsilon_i h\left(\varepsilon_i^{-1} x\right)\to h_\infty(x) \quad \mbox{in} ~~ C_{loc}(\R^n).\]

\begin{lem}\label{lem nondegeneracy}
  $\{\Phi_\infty>0\}=\{t>h_\infty(x)\}$.
\end{lem}
\begin{proof}
Because $u\geq 1-b_0$ in $\{t>h(x)\}$, by Lemma \ref{lem A.1} we get
\[1-u(x)\leq Ce^{-c\left[t-h(x)\right]} \quad \mbox{in} ~~ \{t>h(x)\}.\]
Using the expansion of $g$ near infinity, this is rewritten as
\begin{equation}\label{linear growth}
\Phi(x)\geq c\left[t-h(x)\right]-C  \quad \mbox{in} ~~ \{t>h(x)\}.
\end{equation}
Taking the scaling $\Psi_\varepsilon$ and letting $\varepsilon\to 0$, we obtain
\begin{equation}\label{linear growth 2}
\Phi_\infty(x)\geq c\left[t-h_\infty(x)\right]>0 \quad \mbox{in} ~~\{t>h_\infty(x)\}.
\end{equation}

Finally, because $\Phi\leq g^{-1}(1-b_0)$ in $\{t<h(x)\}$,  $\Phi_\infty\leq 0$ in $\{t<h_\infty(x)\}$.
\end{proof}

\begin{lem}\label{lem unbdd below}
   $h_\infty$ is unbounded from below.
\end{lem}
\begin{proof}
  This is a direct consequence of Proposition \ref{prop dichotomy}.
\end{proof}

\begin{lem}\label{lem Lip constant}
The  Lipschitz constant of $h_\infty$ is at most $\kappa_\ast^{-1}$.
\end{lem}
\begin{proof}
  For any $(x_0,t_0)\in\{\Phi_\infty>0\}$, for all $\varepsilon$ small, $\Phi_\varepsilon(x_0,t_0)\geq \Phi_\infty(x_0,t_0)/2$. By definition, $u(\varepsilon^{-1}x_0,\varepsilon^{-1}t_0)$ is very close to $1$. By Lemma \ref{lem forward Lip}, for any $\delta>0$, there exits a $D(\delta)$ such that $\mathcal{C}^+_{\kappa_\ast-\delta}(\varepsilon^{-1} x_0, \varepsilon^{-1}t_0+D(\delta))\subset\{u>1-b_2\}$. A scaling of this gives
  $\mathcal{C}^+_{\kappa_\ast-\delta}(x_0,t_0+D(\delta)\varepsilon)\subset \{\Phi_\varepsilon>0\}$.
  Letting $\varepsilon\to0$ and then $\delta\to 0$, with the help of \eqref{linear growth 2}, we deduce that $\mathcal{C}^+_{\kappa_\ast}(x_0,t_0)\subset \{\Phi_\infty> 0\}$. This implies that the Lipschitz constant of $h_\infty$ is at most $\kappa_\ast^{-1}$.
\end{proof}

Finally, under the assumption of Theorem \ref{main result 4},  the following non-degeneracy condition in $\Omega_\infty^-$ holds.
\begin{prop}\label{prop nondegeneracy}
In $\Omega_\infty^-$,
  \begin{equation}\label{linear growth 3}
     \Phi_\infty(x,t)\leq c\left[t-h_\infty(x)\right]<0.
  \end{equation}
\end{prop}
This can be proved by scaling Proposition \ref{prop negative part}.

\section{Geometric motion: Proof of Theorem \ref{thm blowing down limit}}\label{sec geometric motion}
\setcounter{equation}{0}

In this section we prove Theorem \ref{thm blowing down limit}. This theorem does not follow directly  from the result on front motion law established in \cite{Barles1993front}, although it can be reduced to that one by constructing a suitable comparison function. The main reason is,  for entire solutions of \eqref{eqn}, it is not clear whether $|\nabla\Phi|\leq 1$ everywhere or not. (We believe this is not true in general.)
\begin{proof}[Proof of Theorem \ref{thm blowing down limit}]
We divide the proof into two steps, verifying the sub- and sup-solution property respectively.

{\bf Step 1.} For any $\varphi\in C^1(\R^n)$ satisfying $\varphi\geq h_\infty$ and $\varphi=h_\infty$ at one point, say the origin $0$, we want to show that $|\nabla\varphi(0)|\leq \kappa_\ast^{-1}$.

Assume by the contrary, there exists $\delta>0$ such that
\begin{equation}\label{absurd assumption 3}
 |\nabla\varphi(0)|=\left(\kappa_\ast-3\delta\right)^{-1}.
\end{equation}
The tangent plane of $\{t=\varphi(x)\}$ at $(0,0)$ is $\left\{\left(\kappa_\ast-3\delta\right)t=-x\cdot\xi\right\}$, where $\xi:=\nabla\varphi(0)/|\nabla\varphi(0)|$. Since $h_\infty\leq\varphi$, we find three small constants $\rho>0$, $t_0<0$ and $\sigma>0$ such that
\[\Phi_\infty(x,t_0)\geq \sigma \quad \mbox{ in} \quad \mathcal{D}:=\left\{x\cdot \xi\geq -\left(\kappa_\ast-2\delta\right)t_0\right\}\cap B_\rho(0).\]

For all $\varepsilon$ small, consider the Cauchy problem
\[
    \left\{
\begin{aligned}
& \partial_t w-\Delta w=f(w)  \quad & \mbox{in } \R^n\times(t_0/\varepsilon,+\infty),\\
&w(x,t_0/\varepsilon)=\left(1-e^{-c\sigma\varepsilon^{-1}\mbox{dist}\left(\varepsilon x,\partial\mathcal{D}\right)}\right)\chi_{\mathcal{D}/\varepsilon}(x),
\end{aligned}\right.
\]
where $\chi_{\mathcal{D}/\varepsilon}$ denotes the characteristic function of $\mathcal{D}/\varepsilon$.

As in Lemma \ref{lem propagation to 1} (or by the motion law for front propagation starting from $\partial(\mathcal{D}\cap B_\rho)$, see \cite[Main Theorem]{Barles1992front} or \cite[Theorem 9.1]{Barles1993front}), we get
\[ w(x,0)\geq 1-b_0 \quad \mbox{in} \quad  \left\{x\cdot\xi\geq \frac{\delta t_0}{\varepsilon}\right\}\cap B_{\frac{\rho+(\kappa_\ast-\delta)t_0}{\varepsilon}}(0).\]
By comparison principle, $u\geq w$ in $\R^n\times[t_0/\varepsilon,0]$.
After a scaling, with the help of \eqref{linear growth 2},   we get
\[\Phi_\infty> 0 \quad \mbox{in} \quad  \left\{x\cdot\xi\geq  \delta t_0 \right\}\cap B_{\rho+(\kappa_\ast-\delta)t_0}(0).\]
In particular, $(0,0)$ is an interior point of $\{\Phi_\infty>0\}$. This is a contradiction.

{\bf Step 2.} In the same way, we can show that for any $\varphi\in C^1(\R^n)$ satisfying $\varphi\leq h_\infty$ and $\varphi(0)=h_\infty(0)$, $|\nabla\varphi(0)|\geq \kappa_\ast^{-1}$.

Assume this is not true, that is, there exists a $\delta>0$ such that $|\nabla\varphi(0)|=\left( \kappa_\ast+3\delta\right)^{-1}$. The only difference with Step 1 is  the construction of the comparison function. Now we need to consider, for all $\varepsilon$ small,  the Cauchy problem
\[
    \left\{
\begin{aligned}
& \partial_t w-\Delta w=f(w)  \quad & \mbox{in } \R^n\times(t_0/\varepsilon,+\infty),\\
&w(x,t_0)=1-\left(1-e^{-c\sigma\varepsilon^{-1}\mbox{dist}\left(\varepsilon x,\partial\mathcal{D}\right)}\right)\chi_{\mathcal{D}/\varepsilon}(x),
\end{aligned}\right.
\]
where  $\mathcal{D}=\left\{x\cdot\xi\leq -\left(\kappa_\ast+2\delta\right)t_0\right\}\cap B_\rho(0)$.

As in Step 1, by the motion law for front propagation starting from $\partial(\mathcal{D}\cap B_\rho)$ given in \cite[Main Theorem]{Barles1992front} or \cite[Theorem 9.1]{Barles1993front}, we deduce that
\[ w(x,0)\leq b_0 \quad \mbox{in} \quad \left\{x\cdot\xi\leq -\delta \frac{t_0}{\varepsilon}\right\}\cap B_{\frac{\rho+(\kappa_\ast-\delta)t_0}{\varepsilon}}(0).\]
This implies that $(0,0)$ is an interior point of $\{\Phi_\infty<0\}$, which is   a contradiction.
\end{proof}

\section{Representation formula for the blowing down limit}\label{sec representation}
\setcounter{equation}{0}

In this section, we give an explicit representation formula for $\Phi_\infty$. We first consider the forward problem \eqref{H-J eqn 3} in $\Omega_\infty^+:=\{t>h_\infty(x)\}$, and then the backward problem \eqref{H-J eqn 4} in $\Omega_\infty^-:=\{t<h_\infty(x)\}$. The main tool used in this section is the generalized characteristics associated to $\Phi_\infty$. We will follow closely the treatment in Cannarsa, Mazzola and Sinestrari \cite{Cannarsa2015global}.

\subsection{Forward problem} \label{subsec positive part}
This subsection is devoted to the forward problem \eqref{H-J eqn 3} in $\Omega_\infty^+$. We first notice the following pointwise monotonicity relation.
\begin{lem}\label{lem pointwise monotonicity}
  For any $(x,t),(y,s)\in\Omega_\infty^+$ with $t>s$, if the segment connecting $(y,s)$ and $(x,t)$ is contained in $\Omega_\infty^+$, then
\begin{equation}\label{pointwise monototinicity}
 \Phi_\infty(x,t)\leq\Phi_\infty(y,s)+\left(\kappa_\ast+\beta_+\right)(t-s)+\frac{|x-y|^2}{4\beta_+(t-s)}.
\end{equation}
\end{lem}
\begin{proof}
 Since $\Phi_\infty$ is Lipschitz, it is differentiable a.e. and satisfies \eqref{H-J eqn 3} a.e. in $\Omega_\infty^+$.
By avoiding a zero measure set, we may  assume for a.e. $\tau\in[0,1]$, $\Phi_\infty$ is differentiable at the point $X(\tau):=((1-\tau)y+\tau x, (1-\tau)s+\tau t)$. (The general case follows by an approximation using the continuity of $\Phi_\infty$.)

Then we have
  \begin{eqnarray*}
    \frac{d}{d\tau}\Phi_\infty\left(X(\tau)\right)
     &=& \partial_t\Phi\left(X(\tau)\right)\left(t-s\right)
     +\nabla\Phi\left(X(\tau)\right)\cdot(x-y) \\
    &=& \left(\kappa_\ast+\beta_+\right)\left(t-s\right)
     -\beta_+|\nabla\Phi\left(X(\tau)\right)|^2\left(t-s\right)\\
    &&+\nabla\Phi\left(X(\tau)\right)\cdot(x-y)\\
    &\leq&\left(\kappa_\ast+\beta_+\right)\left(t-s\right)+\frac{|x-y|^2}{4\beta_+\left(t-s\right)}.
  \end{eqnarray*}
Integrating this inequality in $\tau$, we obtain \eqref{pointwise monototinicity}.
\end{proof}

Next we establish a localized Hopf-Lax formula for $\Phi_\infty$.
\begin{lem}[Localized Hopf-Lax formula I]\label{lem localized Hopf-Lax 1}
  There exists a constant $K$ depending only on  the Lipschitz constant of $\Phi_\infty$ so that the following holds. For any $(x,t)\in\Omega_\infty^+$, there exists an $\varepsilon>0$ such that $B_{K\varepsilon}(x)\times(t-\varepsilon,t)\subset\Omega_\infty^+$, and
 \begin{equation}\label{localized Hopf-Lax 1}
 \Phi_\infty(x,t)=\min_{y\in B_{K\varepsilon}(x)}\left[\Phi_\infty(y,t-\varepsilon)+\left(\kappa_\ast+\beta_+\right)\varepsilon+\frac{|x-y|^2}{4\beta_+\varepsilon}\right].
\end{equation}
\end{lem}
\begin{proof}
Denote $Q:=B_{K\varepsilon}(x)\times(t-\varepsilon,t)$. By Lemma \ref{lem pointwise monotonicity}, we can apply  Lions \cite[Theorem 10.1 and Theorem 11.1]{Lions1982book} to deduce that
\begin{equation}\label{local representation 1}
 \Phi_\infty(x,t)=\inf_{(y,s)\in \partial^pQ}\left[\Phi_\infty(y,s)+\left(\kappa_\ast+\beta_+\right)(t-s)+\frac{|x-y|^2}{4\beta_+(t-s)}\right].
\end{equation}
If $K$ is large enough (compared to the Lipschitz constant of $\Phi_\infty$), for any $(y,s)\in \partial B_{K\varepsilon}(x)\times[t-\varepsilon,t)$, we have
\[\Phi_\infty(y,s)+\left(\kappa_\ast+\beta_+\right)(t-s)+\frac{|x-y|^2}{4\beta_+(t-s)}> \Phi_\infty(x,t-\varepsilon)+\left(\kappa_\ast+\beta_+\right)\varepsilon.\]
Therefore the infimum in \eqref{local representation 1} is attained in the interior of $B_{K\varepsilon}(x)\times\{t-\varepsilon\}$, and it must be a minimum.
\end{proof}

Now we use this localized Hopf-Lax formula to study backward characteristic curves of $\Phi_\infty$. We will restrict our attention to differentiable points. Take a point $(x_0,t_0)\in\Omega_\infty^+$ so that $\Phi_\infty(\cdot,t_0)$ is differentiable at $x_0$. Denote $p_0:=\nabla\Phi_\infty(x_0,t_0)$.

By Lemma \ref{lem localized Hopf-Lax 1}, for any $s<t_0$ sufficiently close to $t_0$, there exists a point $(x(s),s)\in \mathcal{C}^-_{K}(x_0,t_0)\cap\Omega_\infty^+$ such that
\begin{equation}\label{min pt}
 \Phi_\infty(x_0,t_0)= \Phi_\infty(x(s),s)+\left(\kappa_\ast+\beta_+\right)(t_0-s)+\frac{|x_0-x(s)|^2}{4\beta_+(t_0-s)}.
\end{equation}

\begin{lem}\label{lem form of characteristic}
Under the above setting, we have
\begin{equation}\label{representation of min pt}
 x(s)=x_0-2\beta_+(t_0-s)p_0.
\end{equation}
\end{lem}
\begin{proof}
By Lemma \ref{lem localized Hopf-Lax 1}, for any $x$ close to $x_0$,
\begin{equation}\label{5.1.1}
  \Phi_\infty(x,t_0)\geq \Phi_\infty( x(s),s)+\left(\kappa_\ast+\beta_+\right)(t_0-s)+\frac{|x_0- x(s)|^2}{4\beta_+(t_0-s)}.
\end{equation}
Subtracting \eqref{min pt} from \eqref{5.1.1} leads to
\[\Phi_\infty(x,t_0)-\Phi_\infty(x_0,t_0)\geq \frac{x+x_0-2 x(s)}{4\beta_+(t_0-s)}\cdot(x-x_0).\]
On the other hand, because $\Phi_\infty(\cdot,t_0)$ is differentiable at $x_0$, we have
\[\Phi_\infty(x,t_0)-\Phi_\infty(x_0,t_0)=p\cdot (x-x_0)+o\left(|x-x_0|\right).\]
These two relations hold for any $x$ sufficiently close to $x_0$, so
\[p_0=\frac{x_0- x(s)}{2\beta_+(t_0-s)}. \qedhere\]
\end{proof}
\begin{coro}
  The minimum in \eqref{localized Hopf-Lax 1} is attained at a unique point.
\end{coro}
The curve
\[\{(x(s),s): ~~~ x(s)=x_0-2\beta_+(t_0-s)p_0, ~~ s\leq t_0\}\]
is   the backward characteristic curve of $\Phi_\infty$ starting from $(x_0,t_0)$.

\begin{lem}\label{lem propagation of differentiability}
  Under the above settings, $\Phi_\infty(\cdot,s)$ is differentiable at $x(s)$. Moreover,
  \begin{equation}\label{propagation of differentiability}
    \nabla\Phi_\infty(x(s),s)=p_0.
  \end{equation}
\end{lem}
\begin{proof}
Because $x(s)$ attains the minimum in \eqref{local representation 1}, for any $z$ sufficiently close to $x(s)$, we have
\begin{eqnarray*}
   && \Phi_\infty(x(s),s)+\left(\kappa_\ast+\beta_+\right)(t_0-s)+\frac{|x_0-x(s)|^2}{4\beta_+(t_0-s)} \\
   &\leq& \Phi_\infty(z,s)+\left(\kappa_\ast+\beta_+\right)(t_0-s)+\frac{|x_0-z|^2}{4\beta_+(t_0-s)}.
\end{eqnarray*}
After simplification, this is
\[\Phi_\infty(z,s)\geq \Phi_\infty(x(s),s)+\frac{x_0-x(s)}{2\beta_+(t_0-s)}\cdot\left[z-x(s)\right]+O\left(|z-x(s)|^2\right).\]
Since $\Phi_\infty(\cdot,s)$ is semi-concave, this inequality implies that $\Phi_\infty(\cdot,s)$ is differentiable at $x(s)$, and its gradient is given by \eqref{propagation of differentiability}.
\end{proof}
By  Lemma \ref{lem localized Hopf-Lax 1} and Lemma \ref{lem propagation of differentiability}, the characteristic curve can be extended  indefinitely in the  backward direction, unless it hits the boundary $\partial\Omega_\infty^+$ in finite time. Now we show that the later case must happen.

\begin{lem}\label{lem hit boundary 1}
 For any $(x_0,t_0)\in\Omega_\infty^+$ with $\Phi(\cdot,t_0)$  differentiable at $x_0$, there exists an $s_0<t_0$ such that
 \[(x_0-2\beta_+(t_0-s_0)p_0,s_0)\in\partial\Omega_\infty^+.\]
\end{lem}
\begin{proof}
If $( x(s),s)=(x_0-2\beta_+(t_0-s)p_0,s)\in\Omega_\infty^+$, by \eqref{representation of min pt}, \eqref{min pt} can be rewritten as
\begin{equation}\label{values along characteristic 1}
  \Phi_\infty(x_0-2\beta_+(t_0-s)p_0,s)=\Phi_\infty(x_0,t_0)-\left(\kappa_\ast+\beta_++\beta_+|p_0|^2\right)(t_0-s).
\end{equation}
Hence there exists an $s_0$ such  that $ \Phi_\infty(x_0-2\beta_+(t_0-s_0)p_0,s_0)=0$ and $\Phi_\infty(x_0-2\beta_+(t_0-s)p_0,s)>0$ for any $s\in(s_0,t_0]$. Because $\Phi_\infty>0$ in $\Omega_\infty^+$ and $\Phi_\infty=0$ on $\partial\Omega_\infty^+$, $(x_0-2\beta_+(t_0-s_0)p_0,s_0)\in\partial\Omega_\infty^+$.
\end{proof}

\begin{lem}\label{lem representation of blowing down limit 1}
For any $(x,t)\in\Omega_\infty^+$,
\[\Phi_\infty(x,t)=\inf_{y\in\{h_\infty<t\}}\left\{\left(\kappa_\ast+\beta_+\right)\left[t-h (y)\right]+\frac{|x-y|^2}{4\beta_+\left[t-h (y)\right]}\right\}.\]
\end{lem}
\begin{proof}
Choosing $(y,s)=(y,h_\infty(y))$ with $h_\infty(y)<t$ in \eqref{pointwise monototinicity} (here we may assume the segment connecting this point and $(x,t)$ is contained in $\Omega_\infty^+$), and then taking infimum over $y$, we  obtain
\begin{equation}\label{5.1}
  \Phi_\infty(x,t)\leq\inf_{y\in\{h_\infty<t\}}\left\{\left(\kappa_\ast+\beta_+\right)\left[t-h (y)\right]+\frac{|x-y|^2}{4\beta_+\left[t-h (y)\right]}\right\}.
\end{equation}

To show that this is an equality,  we assume without loss of generality that  $x$ is a differentiable point of $\Phi_\infty(\cdot,t)$. Then by Lemma \ref{lem hit boundary 1}, in particular, \eqref{values along characteristic 1}, we find that  $y=x_0-2\beta_+(t_0-s_0)p_0$ attains the equality in \eqref{5.1}.
\end{proof}

\subsection{Backward problem}\label{sec negative part}
For the backward problem \eqref{H-J eqn 4}, we still use backward characteristics to determine the form of $\Phi_\infty^-$. The proof is similar to the forward problem, so most results in this subsection will be  stated without proof.

\begin{lem}\label{lem pointwise monotonicity 2}
  For any $(x,t),(y,s)\in\Omega_\infty^-$ with $t>s$,
\begin{equation}\label{pointwise monototinicity 2}
 \Phi_\infty(x,t)\geq\Phi_\infty(y,s)+\left(\kappa_\ast-\beta_-\right)(t-s)-\frac{|x-y|^2}{4\beta_-(t-s)}.
\end{equation}
\end{lem}
Because $\Omega_\infty^-$ is convex (see Remark \ref{rmk representation for level set limit}), the segment connecting $(y,s)$ and $(x,t)$ is always contained in $\Omega_\infty^-$.

In $\Omega_\infty^-$, $\widetilde{\Phi}_\infty:=-\Phi_\infty^-$ is a viscosity solution of
 \begin{equation}\label{H-J eqn 5}
    \partial_t\widetilde{\Phi}_\infty+\beta_-|\nabla\widetilde{\Phi}_\infty|^2+\kappa_\ast-\beta_-=0.
  \end{equation}
Hence we have the following localized Hopf-Lax formula.
\begin{lem}[Localized Hopf-Lax formula  II]\label{lem localized Hopf-Lax 2}
  There exists a constant $K$ depending only on  the Lipschitz constant of $\Phi_\infty$ so that the following holds. For any $(x,t)\in\Omega_\infty^-$, there exists an $\varepsilon>0$ such that $B_{K\varepsilon}(x)\times(t-\varepsilon,t)\subset\Omega_\infty^-$, and
 \begin{equation}\label{localized Hopf-Lax 2}
 \Phi_\infty(x,t)=\max_{y\in B_{K\varepsilon}(x)}\left[\Phi_\infty(y,t-\varepsilon)+\left(\kappa_\ast-\beta_-\right)\varepsilon-\frac{|x-y|^2}{4\beta_-\varepsilon}\right].
\end{equation}
\end{lem}
Take a point $(x_0,t_0)\in\Omega_\infty^+$ so that $\Phi_\infty(\cdot,t_0)$ is differentiable at $x_0$. Denote $p_0:=\nabla\Phi_\infty(x_0,t_0)$.

By Lemma \ref{lem localized Hopf-Lax 2}, for any $s<t_0$ sufficiently close to $t_0$, there exists a point $(x(s),s)\in \mathcal{C}^-_{K}(x_0,t_0)\cap\Omega_\infty^-$ such that
\begin{equation}\label{min pt 2}
 \Phi_\infty(x_0,t_0)= \Phi_\infty(x(s),s)+\left(\kappa_\ast-\beta_-\right)(t_0-s)-\frac{|x_0-x(s)|^2}{4\beta_-(t_0-s)}.
\end{equation}

\begin{lem}\label{lem form of characteristic 2}
Under the above setting, we have
\begin{equation}\label{representation of min pt 2}
x(s)=x_0+2\beta_-(t_0-s)p_0.
\end{equation}
\end{lem}

\begin{lem}\label{lem propagation of differentiability 2}
Under the above setting, $\Phi_\infty(\cdot,s)$ is differentiable at $x(s)$. Moreover,
  \begin{equation}\label{propagation of differentiability 2}
    \nabla\Phi_\infty(x(s),s)=p_0.
  \end{equation}
\end{lem}

By  Lemma \ref{lem localized Hopf-Lax 2} and Lemma \ref{lem propagation of differentiability 2},  the characteristic curve can be extended    indefinitely in the  backward direction,  unless it hits the boundary $\partial\Omega_\infty^-$ in finite time. Now we show that the latter case must happen.

\begin{lem}\label{lem hit boundary 2}
 For any $(x_0,t_0)\in\Omega_\infty^-$ with $\Phi(\cdot,t_0)$ differentiable at $x_0$, there exists an $s_0<t_0$ such that
 \[(x_0+2\beta_+(t_0-s_0)p_0,s_0)\in\partial\Omega_\infty^-.\]
\end{lem}
\begin{proof}
If $(x(s),s)=(x_0+2\beta_-(t_0-s)p_0,s)\in\Omega_\infty^-$, by \eqref{representation of min pt 2}, \eqref{min pt 2} can be rewritten as
\begin{equation}\label{values along characteristic 2}
  \Phi_\infty(x_0+2\beta_-(t_0-s)p_0,s)=\Phi_\infty(x_0,t_0)-\left(\kappa_\ast-\beta_--\beta_-|p_0|^2\right)(t_0-s).
\end{equation}
If $\beta_->\kappa_\ast$, there exists an $s_0$ such  that $ \Phi_\infty(x_0-2\beta_+(t_0-s_0)p_0,s_0)=0$ and $\Phi_\infty(x_0-2\beta_+(t_0-s)p_0,s)<0$ for any $s\in(s_0,t_0]$. Because $\Phi_\infty<0$ in $\Omega_\infty^-$ and $\Phi_\infty=0$ on $\partial\Omega_\infty^-$, $(x_0+2\beta_-(t_0-s_0)p_0,s_0)\in\partial\Omega_\infty^-$.

If $\beta_-=\kappa_\ast$, this is still the case, unless $p_0=0$. However, if $p_0=0$, the characteristic curve is $(x_0,s)$, and \eqref{values along characteristic 2} reads as
\[\Phi_\infty(x_0,s)\equiv \Phi_\infty(x_0,t_0), \quad \mbox{for any } s<t_0.\]
This cannot happen by \eqref{linear growth 3}.
\end{proof}

With these lemmas in hand, similar to Lemma \ref{lem representation of blowing down limit 1}, we get
\begin{lem}\label{lem representation of blowing down limit 2}
For any $(x,t)\in\Omega_\infty^-$,
\[\Phi_\infty(x,t)=\sup_{y\in\{h_\infty<t\}}
\left\{\left(\kappa_\ast-\beta_-\right)\left[t-h_\infty(y)\right]-\frac{|x-y|^2}{4\beta_-\left[t-h_\infty(y)\right]}\right\}.\]
\end{lem}

\begin{rmk}
\begin{itemize}
\item In the above, we use only backward characteristic curves starting from differentiable points. For a non-differentiable point, there could exist many backward characteristic curves emanating from it, see \cite{Cannarsa2015global}.
  \item   In the monostable case, where  $\kappa_\ast/2\leq \beta_-< \kappa_\ast$, the backward characteristic curves could always stay in the domain and do not hit the boundary. We expect the above representation formula still holds in this case, but do not know how to prove it.
  \item The existence of a nontrivial viscosity solution to \eqref{H-J eqn 4} imposes some restrictions on  the domain $\{t<h_\infty(x)\}$. The following question seems to be interesting, and as far as the author knows, has not been explored in the literature: under what conditions on the domain $\{t<h_\infty(x)\}$, can we prove the nonexistence of viscosity solution of an Hamilton-Jacobi equation?  We may ask the same question for the implication of the existence of globally Lipschitz viscosity solutions.
\end{itemize}

\end{rmk}

\section{Characterization of minimal speed: Proof of Theorem \ref{thm minimal speed}}\label{sec minimal speed}
\setcounter{equation}{0}

In this section we consider travelling wave equation \eqref{travelling wave eqn}.

Denote the constants
\[K_+:=\sqrt{1+\frac{\kappa_\ast}{\beta_+}+\frac{\kappa^2}{4\beta_+^2}}, \quad  \quad  K_-:=\sqrt{1-\frac{\kappa_\ast}{\beta_-}+\frac{\kappa^2}{4\beta_-^2}}.\]
By abusing notations, we will use the following notations about cones in $\R^n$:
\[\mathcal{C}^+_{\lambda}(x):=\{y: y_n-x_n>\lambda|y^\prime-x^\prime|\}, \quad \mathcal{C}^-_{\lambda}(x):=\{y: y_n-x_n<-\lambda|y^\prime-x^\prime|\}.\]

As in Section \ref{sec blowing down}, set $\Psi:=g^{-1}\circ u$. It satisfies
\begin{equation}\label{distance eqn}
  -\Delta\Psi+\kappa\partial_n\Psi=\kappa_\ast+\frac{g^{\prime\prime}(\Psi)}{g^\prime(\Psi)}\left(|\nabla\Psi|^2-1\right).
\end{equation}

Since this is an elliptic equation, we have the following unconditional, global Lipschitz bound on $\Psi$. This lemma holds once the nonlinearity satisfies $f(0)=f(1)=0$, no matter whether it is monostable, combustion or bistable.
\begin{lem}\label{lem gradient bound 2}
  There exists a universal constant $C$ such that $|\nabla\Psi|\leq C$ on $\R^n$.
\end{lem}
\begin{proof}
By definition,
  \[\nabla\Psi=\frac{\nabla u}{g^\prime(\Psi)}.\]
Since $g^\prime$ has a positive lower bound on any compact set of $\R$, by the gradient bound on $u$, $|\nabla\Psi|$ is bounded in $\{1/4<u<3/4\}$.

In $\{u<1/4\}$,
\[g^\prime(\Psi)\geq c g(\Psi)=cu.\]
Hence here we have
\[|\nabla\Psi|\leq C\frac{|\nabla u|}{u}\leq C,\]
where the last inequality follows from Harnack inequality and interior gradient estimates applied to \eqref{travelling wave eqn}.

Similarly, in $\{u>3/4\}$,
\[|\nabla\Psi|\leq C\frac{|\nabla u|}{1-u}\leq C.\qedhere\]
\end{proof}
As before, $\Psi$ is still semi-concave.
\begin{lem}[Semi-concavity]\label{semi-concavity 2}
There exists a universal constant $C$ such that for any $x\in\{\Psi>0\}$,
\[\nabla^2 \Psi(x)\leq \frac{C}{\Psi(x)},\]
and  for any $x\in\{\Psi<0\}$,
\[\nabla^2 \Psi(x)\geq \frac{C}{\Psi(x)}.\]
\end{lem}

For each $\varepsilon>0$, let $\Psi_\varepsilon(x):=\varepsilon\Psi(\varepsilon^{-1}x)$, which satisfies
\begin{equation}\label{vanishing viscosity eqn 2}
  -\varepsilon\Delta\Psi_\varepsilon+\kappa\partial_n\Psi_\varepsilon=\kappa_\ast
  +\frac{g^{\prime\prime}(\varepsilon^{-1}\Psi_\varepsilon)}{g^\prime(\varepsilon^{-1}\Psi_\varepsilon)}\left(|\nabla\Psi_\varepsilon|^2-1\right).
\end{equation}
By the uniform Lipschitz bound on $\Psi_\varepsilon$ from Lemma \ref{lem gradient bound 2}, for any sequence $\varepsilon_i\to0$, there exists a subsequence such that
$\Psi_{\varepsilon_i}\to\Psi_\infty$ in $C_{loc}(\R^n)$.  Then standard vanishing viscosity method gives
\begin{lem}\label{blowing down limit 2}
  In the open set $\{\Psi_\infty>0\}$, $\Psi_\infty$ is a viscosity solution of
  \begin{equation}\label{H-J eqn 1}
    \kappa\partial_n\Psi_\infty-\kappa_\ast+\beta_+\left(|\nabla\Psi_\infty|^2-1\right)=0.
  \end{equation}

  In the open set $\{\Psi_\infty<0\}$ (if non-empty), $\Psi_\infty$ is a viscosity solution of
  \begin{equation}\label{H-J eqn 2}
    \kappa\partial_n\Psi_\infty-\kappa_\ast-\beta_-\left(|\nabla\Psi_\infty|^2-1\right)=0.
  \end{equation}
\end{lem}
\begin{rmk}\label{rmk 6.3}
  Equations \eqref{H-J eqn 1} and \eqref{H-J eqn 2} are the corresponding   travelling wave  equations for the time-dependent Hamilton-Jacobi equations \eqref{H-J eqn 3} and \eqref{H-J eqn 4}.
\end{rmk}

Recall that
\[\{v=1-b_0\}=\{x_n=h(x^\prime)\}.\]
As before, we  define the blowing down limit $h_\infty$ from $h$. By Lemma \ref{lem nondegeneracy},
we still have
\[\{\Psi_\infty>0\}=\{x_n>h_\infty(x^\prime)\}.\]
\begin{prop}\label{prop blowing down limit 2}
The Lipschitz constant of $h_\infty$ is at most $\sqrt{\kappa^2/\kappa_\ast^2-1}$. In particular, we must have $\kappa\geq\kappa_\ast$.
 \end{prop}
\begin{rmk}
  Under the assumptions of Theorem \ref{main result 2}, the blowing down limit $h_\infty$ is a viscosity solution of
  \begin{equation}\label{eikonal eqn 2}
   |\nabla h_\infty|^2-\frac{\kappa^2}{\kappa_\ast^2}+1=0 \quad \mbox{in } \quad \R^{n-1}.
  \end{equation}
  This follows from a reduction of Theorem \ref{thm blowing down limit}.
 \end{rmk}
 \begin{proof}[Proof of Proposition \ref{prop blowing down limit 2}]
The blowing down limit of the level set for the entire solution $v(x+\kappa t e_n)$ is the graph
 \[ t= \frac{h_\infty(x^\prime)-x_n}{\kappa}.\]
By  Lemma \ref{lem Lip constant}, its Lipschitz constant is at most $\kappa_\ast^{-1}$.
 \end{proof}

By Lemma \ref{lem semi-concavity}, $\Psi_\varepsilon$ are uniformly semi-concave in any compact set of $\{\Psi_\infty>0\}$. As a consequence, $\Psi_\infty$ is locally semi-concave in this open set. The sup-differential of $\Psi_\infty$ is then well defined at every point in $\{\Psi_\infty>0\}$. Recall that
\[\partial\Psi_\infty(x):=\left\{\xi\in\R^n:  \limsup_{y\to x}\frac{\Psi_\infty(y)-\Psi_\infty(x)-\xi\cdot (y-x)}{|y-x|}\leq 0\right\} \]
 is a compact convex subset of $\R^n$. Because $\Psi_\varepsilon\to\Psi_\infty$ uniformly on any compact set of $\R^n$, by the uniform semi-concavity of $\Psi_\varepsilon$, we deduce that for any $x_\varepsilon\to x_\infty\in\{\Psi_\infty>0\}$,
 \begin{equation}\label{convergence of gradients}
 \mbox{each limit point of} ~~ \nabla\Psi_\varepsilon(x_\varepsilon) ~~ \mbox{as}~  \varepsilon\to0  \in \partial\Psi_\infty(x_\infty).
 \end{equation}
If $\Psi_\infty<0$ in $\{x_n<h_\infty(x^\prime)\}$, the same result holds for the negative part of $\Psi_\infty$, with sup-differentials replaced by sub-differentials.

A reduction of Lemma \ref{lem representation of blowing down limit 1} and Lemma \ref{lem representation of blowing down limit 2} gives
\begin{prop}\label{prop solution of H-J eqn}
\begin{itemize}
  \item  For any $x=(x^\prime,x_n)\in\{\Psi_\infty>0\}$,
  \begin{equation}\label{explicit formulation 6.1}
     \Psi_\infty(x)=  \inf_{y^\prime\in \R^{n-1}}\left[K_+ \sqrt{|x^\prime-y^\prime|^2+(x_n-h_\infty(y^\prime))^2} -\frac{\kappa}{2\beta_+}\left(x_n-h_\infty(y^\prime)\right)\right].
  \end{equation}
  \item Assume $\Psi_\infty<0$ in $\{x_n<h_\infty(x^\prime)\}$. Then for any $x=(x^\prime,x_n)\in\{x_n<h_\infty(x^\prime)\}$,
  \begin{equation}\label{explicit formulation 6.2}
     \Psi_\infty(x)=  -\inf_{y^\prime\in \R^{n-1}}\left[K_- \sqrt{|x^\prime-y^\prime|^2+(x_n-h_\infty(y^\prime))^2} -\frac{\kappa}{2\beta_-}\left(x_n-h_\infty(y^\prime)\right)\right].
  \end{equation}
\end{itemize}
\end{prop}
\begin{rmk}
The representation formula \eqref{explicit formulation 6.1} and \eqref{explicit formulation 6.2} (when $\beta_->\kappa_\ast$), can be proved  directly by rewriting  \eqref{H-J eqn 1} and \eqref{H-J eqn 2} as   eikonal equations. For example, in the case of \eqref{H-J eqn 1}, we can define a norm on $\R^n$, $\|\cdot\|$ so that the corresponding unit ball is  $B_{K_+}\left(0^\prime,-\frac{\kappa}{2\beta_+}\right)$. (This is because  this ball contains the origin as an interior point.) The Hamilton-Jacobi equation \eqref{H-J eqn 1} is equivalent to the eikonal type equation
\begin{equation}\label{eikonal eqn}
   \|\nabla\Psi(x)\|^2-1=0.
\end{equation}
Then we can prove that
\[ \Psi_\infty(x)=  \inf_{y\in \partial\{\Phi_\infty>0\}}\|x-y\|^\ast.\]
Here $\|\cdot\|^\ast$ denotes the dual norm of $\|\cdot\|$.
\end{rmk}

Now we come to
\begin{proof}[Proof of Theorem \ref{thm minimal speed}]
We have shown  that $\kappa\geq\kappa_\ast$ in Proposition \ref{prop blowing down limit 2}. It remains to characterize the $\kappa=\kappa_\ast$ case.

 From Proposition \ref{prop blowing down limit 2}, it is seen that, if $\kappa=\kappa_\ast$, we must have $\nabla h_\infty=0$ a.e. in $\R^{n-1}$. Since $h_\infty(0)=0$, we get
  \begin{equation}\label{flat limit}
   h_\infty\equiv 0   \quad \mbox{in } ~~ \R^{n-1}.
  \end{equation}
This holds for any blowing down limit $h_\infty$, so the blowing down limit is unique.

Substituting \eqref{flat limit} into \eqref{explicit formulation 6.1} and \eqref{explicit formulation 6.2}, by noting that
$\kappa=\kappa_\ast$ implies
\[K_+=1+\frac{\kappa_\ast}{2\beta_+},\quad K_-=1-\frac{\kappa_\ast}{2\beta_-},\]
we deduce that
\begin{equation}\label{1d limit}
  \Psi_\infty(x)\equiv x_n \quad \mbox{in }~~ \R^n.
\end{equation}

{\bf Claim.} For any $\varepsilon>0$, there exists an $L(\varepsilon)>0$ such that
\[|\nabla^\prime\Psi|\leq \varepsilon \partial_n\Psi \quad \mbox{in }~~ \{|\Psi|\geq L(\varepsilon)\}.\]

By this claim,  similar to the proof of Theorem \ref{main result 2} (or as in the proof of Gibbons conjecture in \cite{Farina1999symmetry,Hamel2000Gibbons}), applying the sliding method we deduce that
\[|\nabla^\prime\Psi|\leq \varepsilon \partial_n\Psi \quad \mbox{in }~~ \R^n.\]
Letting $\varepsilon\to0$, we deduce that $\nabla^\prime\Psi\equiv 0$, or equivalently, $v$ is a function of $x_n$ only. Because we have \eqref{assumption} and $\sup_{\R^n}v=1$, by   the uniqueness of $g$, we find a constant $t\in\R$ such that
\[v(x)\equiv g(x_n+t)  \quad \mbox{in }~~ \R^n.\]

{\bf Proof of the claim.} Assume by the contrary, there exists a sequence of points $x_i$ with
\[\varepsilon_i^{-1}:=|\Psi(x_i)|\to +\infty,\]
but
\begin{equation}\label{absurd assumption 4}
  |\nabla^\prime\Psi(x_i)|\geq \varepsilon \partial_n\Psi(x_i).
\end{equation}
Let $\Psi_i(x):=\varepsilon_i\Psi(\varepsilon_i^{-1}x)$. Combining \eqref{convergence of gradients} and \eqref{absurd assumption 4} together leads to a contradiction with \eqref{1d limit}.
\end{proof}

{\bf Acknowledgement.} The author's research was supported by the National Natural Science Foundation of China No.~11871381 and No. 11631011. He would like to thank Professor Yihong Du for pointing out an error in the draft version, and to Professor Rui Huang for a discussion several years ago.

%\bibliography{Travellingwaves}
%\bibliographystyle{plain}

\end{document}